\documentclass[a4,11pt]{amsart}
\usepackage{amssymb,amsmath,amsthm,txfonts}
\usepackage{cite}
\usepackage{amsrefs}

\newtheorem{thm}{Theorem}[section]
\newtheorem{lem}[thm]{Lemma}
\newtheorem{cor}[thm]{Corollary}
\newtheorem{prop}[thm]{Proposition}
\theoremstyle{definition}

\theoremstyle{remark}
\newtheorem{rem}[thm]{\textbf{Remark}}
\newtheorem{rems}[thm]{\textbf{Remarks}}

 \makeatletter
    
    \@addtoreset{equation}{section}
  \makeatother

\makeatletter
      \def\@makefnmark{%
         \leavevmode
            \raise.9ex\hbox{\check@mathfonts
                \fontsize\sf@size\z@\normalfont%
                            \@thefnmark}%
       }
      \makeatother

\newcommand{\h}{\dot{H}^{1}}
\renewcommand{\H}{\dot{H}^{1}_{0}}
\newcommand{\p}{\mathbb{P}}

\newcommand{\D}{\textrm{div}}

\newcommand{\dd}{\textrm{d}}

\begin{document}

\title[]{Global well-posedness of the two-dimensional exterior Navier-Stokes equations for non-decaying data}
\author[]{Ken Abe}
\date{}
\address[K. Abe]{Department of Mathematics, Graduate School of Science, Osaka City University, 3-3-138 Sugimoto, Sumiyoshi-ku Osaka, 558-8585, Japan}
\email{kabe@sci.osaka-cu.ac.jp}
\date{\today}

\maketitle


\begin{abstract}
We prove global well-posedness of the two-dimensional exterior Navier-Stokes equations for bounded initial data with a finite Dirichlet integral, subject to the non-slip boundary condition. As an application, we construct global solutions for asymptotically constant initial data and arbitrary large Reynolds numbers. 
\end{abstract}

\vspace{15pt}

\section{Introduction}
\vspace{10pt}
We consider the two-dimensional Navier-Stokes equations in an  exterior domain $\Omega\subset \mathbb{R}^{2}$:\\
\begin{equation*}
\begin{aligned}
\partial_t u-\Delta{u}+u\cdot \nabla u+\nabla{p}&= 0 \quad \textrm{in}\quad \Omega\times(0,\infty),    \\
\D\ u&=0\quad \textrm{in}\quad \Omega\times(0,\infty),    \\
u&=0\quad \textrm{on}\quad \partial\Omega\times (0,\infty) ,  \\
u&=u_0\quad\hspace{-4pt} \textrm{on}\quad  \Omega\times \{t=0\}.  
\end{aligned}
\tag{1.1}
\end{equation*}\\
It is well known that the two-dimensional exterior Navier-Stokes equations are globally well-posed for initial data with finite energy \cite{Leray1933}, \cite{Leray1934b}, \cite{La59}. However, global solvability is unknown in general for initial data with infinite energy. An example of infinite energy solutions is a stationary solution with a finite Dirichlet integral, called $D$-solution \cite{Leray1933}. It is known that $D$-solutions are bounded in $\Omega$ and asymptotically constant as $|x|\to\infty$; see Remark 1.2. In other words, $D$-solutions are elements of $L^{\infty}\cap \dot{H}^{1}_{0}$, where $\h$ denotes the homogeneous $L^{2}$-Sobolev space and $\dot{H}^{1}_{0}$ denotes the space of all functions in $\h$, vanishing on $\partial\Omega$.

The purpose of this paper is to establish the global solvability of (1.1) for non-decaying initial data $u_0\in L^{\infty}_{\sigma}\cap \h_0$. We set the solenoidal $L^{\infty}$-space,  

\begin{equation*}
L^{\infty}_{\sigma}(\Omega)
=\left\{ u\in L^{\infty}(\Omega)\ \Bigg|\ \int_{\Omega}u\cdot \nabla \varphi\dd x=0 \quad \textrm{for}\ \varphi\in \hat{W}^{1,1}(\Omega) \right\},  
\end{equation*}\\
by the homogeneous Sobolev space $\hat{W}^{1,1}(\Omega)=\{\varphi\in L^{1}_{\textrm{loc}}(\Omega)\ |\ \nabla \varphi\in L^{1}(\Omega)\ \}$. For exterior domains, the space $L^{\infty}_{\sigma}$ agrees with the space of all bounded divergence-free vector fields, whose normal trace is vanishing on $\partial\Omega$ \cite{AG2}. When $\Omega$ is the whole space, the global solvability of (1.1) is established in \cite{GMS} for merely bounded initial data $u_0\in L^{\infty}_{\sigma}$. Since the vortex stretching is absent for the two-dimensional vorticity equation, 

\begin{align*}
\partial_t \omega+u\cdot\nabla \omega-\Delta\omega&=0\qquad \textrm{in}\ \mathbb{R}^{2}\times (0,\infty),\\
\omega&=\omega_0\hspace{15pt}   \textrm{on}\ \mathbb{R}^{2}\times\{t=0\},
\end{align*}\\
the global a priori estimate 

\begin{align*}
||\omega||_{L^{\infty}(\mathbb{R}^{2})}\leq ||\omega_0||_{L^{\infty}(\mathbb{R}^{2})}\quad t>0,
\end{align*}\\
holds for $\omega=\partial_1 u^{2}-\partial_2 u^{1}$ by the maximum principle. (One may assume that the initial vorticity is bounded by the regularizing effect on $L^{\infty}$ \cite{GIM}.) The vorticity estimate plays a crucial role in order to construct global mild solutions on $L^{\infty}_{\sigma}$. 

When $\Omega$ is a domain with boundaries, vorticity propagation is more involved. The global vorticity estimate may not hold and it is unknown whether the problem (1.1) is globally well-posed on $L^{\infty}_{\sigma}$. For exterior domains, local solvability on $L^{\infty}_{\sigma}$ is recently established in \cite{A5} (see also \cite{A4}).

For initial data with finite energy, there is a global bound for vorticity. Since local solutions exist in $(0,T]$ for some $T>0$ and satisfies

\begin{align*}
\sup_{0<t\leq T}t^{\frac{1}{2}}\Big\{||u||_{L^{\infty}}(t)+||\nabla u||_{L^{2}}(t)\Big\}\leq C||u_0||_{L^{2}},
\end{align*}\\
by the energy equality,

\begin{align*}
\int_{\Omega}|u|^{2}\dd x+2\int_{0}^{t}\int_{\Omega}|\nabla u|^{2}\dd x\dd s= \int_{\Omega}|u_0|^{2}\dd x\quad t>0,
\end{align*}\\ 
local solutions are continued for all $t\geq T$ (see, e.g., \cite{KO}). In other words, global solutions exist, provided that initial data is  decaying sufficiently fast as $|x|\to\infty$. Since local solutions are bounded and with a finite Dirichlet integral for each $t>0$, one may assume the regularity condition $u_0\in L^{\infty}\cap \dot{H}^{1}_{0}$ for initial data with finite energy. On the other hand, a finite Dirichlet integral does not imply decay at space infinity. Thus the condition $u_0\in L^{\infty}\cap \dot{H}^{1}_{0}$ can be viewed as an extension of finite energy.

Our goal is to construct a global mild solution for $u_0\in L^{\infty}_{\sigma}\cap \dot{H}^{1}_{0}$. As explained later in the introduction, the space $L^{\infty}_{\sigma}\cap \dot{H}^{1}_{0}$ includes asymptotically constant vector fields. We prove global solvability of (1.1) for non-decaying initial data $u_0\in L^{\infty}_{\sigma}\cap \dot{H}^{1}_{0}$, and deduce existence of asymptotically constant solutions for \textit{arbitrary large} Reynolds numbers. To state a result, let $S(t)$ denote the Stokes semigroup. Let $\p$ denote the Helmholtz projection. We construct global mild solutions of the form

\begin{equation*}
u(t)=S(t)u_{0}-\int_{0}^{t}S(t-s)\p  (u\cdot \nabla u)(s)ds.     \tag{1.2}
\end{equation*}\\
Let $C([0,T]; X)$ (resp. $C_{w}([0,T]; X)$) denote the space of all continuous (resp. weakly continuous) functions from $[0,T]$ to a Banach space $X$. When $X=L^{\infty}$, the space $C_{w}([0,T]; L^{\infty})$ denotes the space of all weakly-star continuous functions. Since the Stokes semigroup is an analytic semigroup on $L^{\infty}_{\sigma}$ \cite{AG2}, the first term is defined for non-decaying initial data $u_0\in L^{\infty}_{\sigma}\cap \h_0$. The second term is defined by the Helmholtz projection on $L^{2}$ for $u\in C_{w}([0,\infty); L^{\infty})$ satisfying $\nabla u\in C_{w}([0,\infty); L^{2})$. The mild solutions constructed in this paper are sufficiently smooth and satisfy (1.1) in a suitable sense (see Remarks 5.1). The main result of this paper is the following:

\vspace{10pt}

\begin{thm}
Let $\Omega$ be an exterior domain with $C^{3}$-boundary in $\mathbb{R}^{2}$. Assume that $u_0\in L^{\infty}_{\sigma}\cap \dot{H}^{1}_{0}$. Then, there exists a unique mild solution $u\in C_{w}([0,\infty); L^{\infty})$ of (1.1) satisfying $\nabla u\in C_{w}([0,\infty); L^{2})$.
\end{thm}

\vspace{5pt}

\begin{rem}
($D$-solutions) We constructed global solutions of (1.1) in $L^{\infty}\cap \h$. The space $L^{\infty}\cap\h$ is a natural space for studying non-decaying solutions. So far, various stationary solutions have been constructed in $L^{\infty}\cap \h$, while the existence of global solutions was unknown for the non-stationary problem (1.1). Stationary solutions of (1.1) were first constructed by Leray \cite{Leray1933} based on an approximation as $R\to\infty$ of the problem 

\begin{align*}
-\Delta u_{R}+u_{R}\cdot \nabla u_{R}+\nabla p_{R}&=0\quad \textrm{in}\ \Omega_R,\\
\D\ u_R&=0\quad \textrm{in}\ \Omega_R,\\
u_{R}&=0\quad \textrm{on}\ \partial\Omega,\\
u_{R}&=u_{\infty}\hspace{4pt} \textrm{on}\ \{|x|=R\},
\end{align*}\\
for $\Omega_{R}=\Omega\cap \{|x|<R\}$ and the constant $u_{\infty}\in \mathbb{R}^{2}$ (see also \cite[Chapter 5, Theorem 5]{Lady61}). The solutions constructed by the Leray's method are with a finite Dirichlet integral and called Leray's solution. Later on, stationary solutions with a finite Dirichlet integral are constructed also by the Galerkin method in \cite[Theorem 3.2]{Fujita61} (\cite[Theorem X.4.1]{Gal}). We refer to any stationary solutions of (1.1) with a finite Dirichlet integral as $D$-solution. Note that a finite Dirichlet integral does not imply a global bound as $|x|\to\infty$ (e.g., $u=(\log|x|)^{\alpha}$, $0<\alpha<1/2$). It is proved by Gilbarg and Weinberger \cite{GW} (\cite{GW2}) that Leray's solutions are bounded in $\Omega$ and converge to some constant $\overline{u}_{\infty}$ in the sense that $\int_{0}^{2\pi}|u(re_{r})-\overline{u}_{\infty}|^{2}\dd \theta\to 0$ as $r\to\infty$, where $(r,\theta)$ is the polar coordinate and $e_{r}=(\cos\theta,\sin\theta)$. Moreover, every $D$-solutions are bounded and asymptotically constant in the above sense \cite[Theorem 12]{Amick88}. Thus, $D$-solutions are elements of  $L^{\infty}_{\sigma}\cap \dot{H}^{1}_{0}$. 

When the constant $u_{\infty}$ is sufficiently small ($u_{\infty}\neq 0$), Finn and Smith \cite[Corollary 4.2]{FS} constructed unique stationary solutions of (1.1) satisfying $u-u_{\infty}=O(|x|^{-\varepsilon-1/4})$ as $|x|\to\infty$ for some $\varepsilon>0$. The solutions with this decay rate are called $PR$-solution \cite[Section 4]{Smith65}. It is known that $PR$-solutions have the faster decay $u-u_{\infty}=O(|x|^{-1/2})$ as $|x|\to\infty$ \cite{Smith65}. Since the $PR$-solutions are with a finite Dirichlet integral \cite[Lemma 5.2]{FS}, they are element of $L^{\infty}_{\sigma}\cap \dot{H}^{1}_{0}$. See \cite{Gal04}, \cite{KPR} for more information about stationary solutions. Note that for large Reynolds numbers, existence of stationary solutions is a long standing open question. Here, we regard the constant $|u_{\infty}|$ as the Reynolds number. As stated below in Theorem 1.3, Theorem 1.1 implies existence of global solutions of (1.1) for arbitrary large Reynolds numbers.
\end{rem}

\vspace{10pt}

From a physical point of view, when the Reynolds number is large, a viscous fluid can vary behind the obstacle periodically or irregularly. It is natural to consider the non-stationary problem (1.1) for studying non-decaying flows of large Reynolds numbers. So far, existence of global solutions of (1.1) was unknown for asymptotically constant initial data. One may construct global solutions of (1.1) for small Reynolds numbers by subtracting a $PR$-solution from $u$ and reducing the problem to decaying initial data with finite energy. We deduce existence of global solutions from Theorem 1.1 without constructing stationary solutions, and obtain asymptotically constant solutions for arbitrary large Reynolds numbers. 

To state a result, let $BUC$ denote the space of all bounded uniformly continuous functions in $\overline{\Omega}$. Let $BUC_{\sigma}$ denote the space of all solenoidal vector fields in $BUC$, vanishing on $\partial\Omega$. We consider asymptotically constant initial data $u_0\in BUC_{\sigma}\cap \h_{0}$ such that 

\begin{align*}
\lim_{R\to\infty}\sup_{|x|\geq R}|u_0(x)-u_{\infty}|=0.
\end{align*}\\
For simplicity of the notation, we shall denote by $u_0\to u_{\infty}$ as $|x|\to\infty$. Since the Stokes semigroup $S(t)u_0$ is asymptotically constant as $|x|\to\infty$ for such the initial data, sending $|x|\to\infty$ to (1.2) implies that mild solutions are also asymptotically constant. From Theorem 1.1, we deduce the following:

\vspace{10pt}

\begin{thm}
Assume that $u_\infty\in \mathbb{R}^{2}$ and $u_0\in BUC_{\sigma}\cap\h_{0}$ satisfy $u_0\to u_{\infty}$ as $|x|\to\infty$. Then, the mild solution $u\in C([0,\infty); BUC)$ satisfies $u\to u_{\infty}$ as $|x|\to\infty$ for each $t\geq 0$.
\end{thm}

\vspace{5pt}

For the two-dimensional Euler equations, global solvability of the exterior problem is proved in \cite{Kikuchi} for asymptotically constant initial data. In the paper, global solutions are constructed by Schauder's fixed point theorem for sufficiently smooth and asymptotically constant initial data, satisfying the decay condition of vorticity 

\begin{align*}
\int_{\Omega}(1+|x|^{\theta})|\omega_{0}(x)|\dd x<\infty,\quad\textrm{for some}\ \theta>0.
\end{align*}\\
For the two-dimensional ideal flows, vorticity moves along a stream line and does not increase. Hence by using a stream line, the vorticity can be regarded as a fixed point of some map between a space of bounded functions (see also \cite{Kato1967}). 

On the other hand, for viscous flows we expect increase of vorticity since there exist boundary layers. As we have seen in the introduction, for decaying initial data sufficiently fast as $|x|\to\infty$, vorticity is globally bounded in $L^{2}$ by the energy inequality, and solutions are sufficiently smooth for all $t>0$ even if new vortices are produced from the boundaries. For slowly decaying or non-decaying initial data, we expect that stronger vorticity is produced from the boundaries. Indeed, for viscous flows of large Reynolds numbers, separations of boundary layers can produce turbulence behind the obstacle $\Omega^{c}$ (depending on shapes of the body $\Omega^{c}$). See \cite{LL} for example. So far, a global vorticity bound was unknown for such non-decaying flows. We proved that vorticity of non-decaying flows in $L^{\infty}\cap \h$ are globally bounded in $L^{2}$ for each $t>0$ and solutions are sufficiently smooth for all $t>0$ for arbitrary large Reynolds numbers. Note that Theorem 1.3 holds for exterior domains of class $C^{3}$, independently of shapes of the body $\Omega^{c}$.

Theorem 1.3 implies existence of global solutions for general viscosities $\nu>0$ and arbitrary Reynolds numbers $R>0$. Here, we set the Reynolds number by

\begin{align*}
R=\frac{|u_{\infty}|}{\nu}\textrm{diam}\ \Omega^{c}.
\end{align*}\\
In fact, for solutions $(u,p)$ constructed in Theorem 1.3, we consider the scaling 

\begin{align*}
u_{\nu}(x,t)=\frac{1}{\nu^{\alpha}}u\Big(\frac{x}{\nu^{1+\alpha}},\frac{t}{\nu^{1+2\alpha}} \Big),\quad p_{\nu}(x,t)=\frac{1}{\nu^{2\alpha}}p\Big(\frac{x}{\nu^{1+\alpha}},\frac{t}{\nu^{1+2\alpha}} \Big)
\end{align*}\\
with some constant $\alpha\in \mathbb{R}$ and obtain the global solution $(u_{\nu},p_{\nu})$ for $\nu>0$ in the exterior domain $\Omega_{\nu}=\nu^{1+\alpha}\Omega$ satisfying $u_{\nu}\to \nu^{-\alpha} u_{\infty}$ as $|x|\to\infty$. The Reynolds number is a dimensionless parameter, used as an important index for characterizing status of viscous flows. One may observe that for fluid flows of large Reynolds numbers, the viscosities $\nu>0$ are relatively small. However, such flows may not be simply understood like ideal flows. Indeed, it is known that small viscosities $\nu>0$ have a significant effect in boundary layers \cite{LL}. Theorem 1.3 implies existence of global solutions for small viscosities $\nu>0$ and large Reynolds numbers $R>0$.

It is an interesting question whether mild solutions approach stationary solutions as time goes to infinity. When $\Omega$ is the whole space, mild solutions for $u_0\in L^{\infty}_{\sigma}\cap \h$ approach constants as $t\to\infty$. Since the vorticity satisfies the global estimate 

\begin{align*}
||\omega||_{L^{p}(\mathbb{R}^{2})}\leq \frac{C}{t^{\frac{1}{2}-\frac{1}{p}}}||\omega_0||_{L^{2}(\mathbb{R}^{2})}\quad t>0,
\end{align*}\\
for $2\leq p\leq \infty$ (e.g., \cite[Chapter 2]{GGS}), the Biot-Savart law implies that 

\begin{align*}
||\nabla u||_{L^{p}(\mathbb{R}^{2})}
\leq C_{p}||\omega||_{L^{p}(\mathbb{R}^{2})}
\to 0\quad \textrm{as}\ t\to\infty\quad \textrm{for}\ 2<p<\infty.
\end{align*}\\
Since stationary solutions of (1.1) with a finite Dirichlet integral in $\mathbb{R}^{2}$ must be constants by the Liouville-type theorem \cite[Theorem 2]{GW}, one can rephrase that the limit as $t\to\infty$ is a trivial stationary solution.

When $\Omega$ is a half space, global solvability of (1.1) is unknown for non-decaying initial data; see \cite{Sl03}, \cite{Mar09}, \cite{BJ}, \cite{A5} for local solvability results. The statement of Theorem 1.1 is valid also for a half space and global mild solutions exist for $u_0\in L^{\infty}_{\sigma}\cap \h_{0}(\mathbb{R}^{2}_{+})$. (For a half space, the condition $u \in \dot{H}^{1}_{0}$ implies a decay as $|x|\to\infty$. See Remarks 6.4 (ii).) It is unknown whether the corresponding Liouville-type theorem holds for a half space since vorticity does not vanish on the boundary. We refer to \cite{Seregin15}, \cite{GHM} for Liouville-type theorems in a half space. Note that for exterior domains, Leray's solutions are indeed non-trivial \cite[Theorem 29]{Amick88}.\\

\vspace{15pt}

Let us sketch the proof of Theorem 1.1. We construct global mild  solutions based on the Stokes semigroup on $L^{\infty}\cap \h$. We set 

\begin{align*}
&|f|_{\dot{H}^{1}}=||\nabla f||_{L^{2}}, \\
&||f||_{L^{\infty}\cap \dot{H}^{1}}=||f||_{L^{\infty}}+|f|_{\dot{H}^{1}}.
\end{align*}\\
The space $L^{\infty}\cap \dot{H}^{1}$ is a Banach space equipped with the norm $||\cdot ||_{L^{\infty}\cap \dot{H}^{1}}$. Note that the norm is homogeneous by the scaling, i.e., $||f_{\lambda}||_{L^{\infty}\cap \h}=\lambda ||f||_{L^{\infty}\cap \h}$ for $f_{\lambda}(x)=\lambda f(\lambda x)$, $\lambda>0$. We first prove the a priori estimate of the Stokes flow 

\begin{align*}
\sup_{0<t\leq T_0}\Big\{\|S(t)u_0\|_{L^{\infty}\cap \h}
+t||A S(t)u_0||_{L^{\infty}\cap \h}\Big\}
\leq C||u_0||_{L^{\infty}\cap \h}  \tag{1.3}
\end{align*}\\
for $u_0\in L^{\infty}_{\sigma}\cap \h_{0}$ and $T_0>0$. Here, $A$ denotes the Stokes operator. The estimate (1.3) implies that the Stokes semigroup is an analytic semigroup on $L^{\infty}_{\sigma}\cap \h_{0}$. When $\Omega=\mathbb{R}^{2}$, the Stokes semigroup agrees with the heat semigroup and the estimate (1.3) holds; see Remarks 2.6 (i). We prove (1.3) for two-dimensional exterior domains by using fractional powers of the Stokes operator on $L^{2}$ and an approximation for $u_0\in L^{\infty}_{\sigma}\cap \h_{0}$. The estimate (1.3) implies the regularity properties
 
\begin{align*}
v,\ t^{\frac{1}{2}}\nabla v\in C_{w}([0,T_0]; L^{\infty}),\quad
\nabla v\in C_{w}([0,T_0]; L^{2}),   \tag{1.4}
\end{align*}\\
for the Stokes flow $v=S(t)u_0$ and $u_0\in L^{\infty}_{\sigma}\cap \h_{0}$.

The second step of the proof is to construct local-in-time solutions of (1.1). In this paper, we study local solvability of the Navier-Stokes system in a general form 

\begin{equation*}
\begin{aligned}
\partial_t w-\Delta{w}+w\cdot \nabla w+v\cdot \nabla w+\nabla{\pi}&= -w\cdot \nabla v-v\cdot \nabla v \quad \textrm{in}\quad \Omega\times(0,T_0),    \\
\D\ w&=0\quad\qquad\qquad\qquad \textrm{in}\quad \Omega\times(0,T_0),    \\
w&=0\qquad\qquad\qquad\quad \textrm{on}\quad \partial\Omega\times (0,T_0) ,  \\
w&=w_0\hspace{70pt}\textrm{on}\quad  \Omega\times \{t=0\}.  
\end{aligned}
\tag{1.5}
\end{equation*}\\
Observe that the system (1.5) agrees with the original system (1.1) for $v\equiv 0$. Moreover, for the Stokes flow $v=S(t)u_0$ and the solution $u$ of (1.1), the difference $w=u-v$ satisfies the system (1.5) for $w_0\equiv 0$. We establish a local solvability of (1.5) on $L^{\infty}\cap\h$ for a solenoidal vector field $v$ satisfying (1.4) and deduce the existence of the local solution $u$ of (1.1).

A crucial step is to extend the local solution $u$ of (1.1) by a global energy estimate. We apply a global energy estimate of the perturbed system (1.5) for $w_0\equiv 0$ of the form

\begin{equation*}
\begin{aligned}
&\int_{\Omega}|w|^{2}\dd x+\int_{0}^{t}\int_{\Omega}|\nabla w|^{2}\dd x\dd s\leq \frac{1}{2}(e^{2N^{2}t}-1)\quad 0\leq t\leq T_0, \\[15pt]
&N=\sup_{0<t\leq T_0}\{||v||_{L^{\infty}\cap \h}+t^{\frac{1}{2}}||\nabla v||_{L^{\infty}}\}.
\end{aligned}
\end{equation*}\\
Since the global energy is bounded for the two-dimensional system (1.5), local solutions of (1.5) for $w_0\equiv 0$ are globally bounded on $L^{\infty}\cap H^{1}$. We prove a global bound on $L^{\infty}\cap H^{1}$ by using an $L^{p}$-blow-up estimate of the system (1.5). Since $w=u-v$ solves the perturbed system (1.5) for $w_0\equiv 0$ and $v=S(t)u_0$, we deduce that $u-v$ is globally bounded on $L^{\infty}\cap H^{1}$ and the local solution $u$ is continued for all $t>0$.

This paper is organized as follows. In Section 2, we prove that the Stokes semigroup is an analytic semigroup on $L^{\infty}_{\sigma}\cap \dot{H}^{1}_{0}$. In Section 3, we establish local solvability of the system (1.5) on $L^{\infty}_{\sigma}\cap \dot{H}^{1}_{0}$. In Section 4, we prove that mild solutions of (1.5) for $w_0\equiv 0$ are globally bounded on $L^{\infty}\cap H^{1}$. In Section 5, we prove Theorem 1.1. In Section 6, we study asymptotic behavior of the Stokes flow as $|x|\to\infty$ and prove Theorem 1.3. In Appendix A, we prove local solvability of (1.5) on $L^{p}$ for $p> 2$ and deduce a blow-up estimate used in Section 4. In Appendix B, we note $L^{p}$-estimates of a fractional power of the Stokes operator in a half space, related to Remarks 2.6 (ii).\\

\vspace{5pt}

After the first draft of this paper is written, the author learned the papers \cite{SawadaTaniuchi}, \cite{Zelik} on large time $L^{\infty}$-estimates of bounded solutions in $\mathbb{R}^{2}$. Although vorticity of the Cauchy problem is uniformly bounded for all $t>0$, a uniform $L^{\infty}$-estimate for velocity is unknown. In \cite{GMS}, Giga et al. gave a double exponential bound for the sup-norm of velocity by using a logarithmic Gronwall's inequality. Their estimate is later improved by Sawada-Taniuchi \cite{SawadaTaniuchi} to a single exponential bound by some approximation argument. More recently, the result is further improved by Zelik \cite{Zelik} to a linear growth estimate by using a uniform local energy. We refer to a lecture note of Gallay \cite{Gallay14} for this uniform local energy estimate. For periodic initial data in one space direction, a uniform $L^{\infty}$-estimate is proved by Gallay-Slijep\v{c}evi\'c \cite{GallayS} (see also \cite{GallayS2}). 

The author learned a recent result of Maremonti-Shimizu \cite{MaremontiShimizu}. In the paper, global solutions in an exterior domain are constructed for merely bounded initial data by a perturbation from $\mathbb{R}^{2}$. See \cite{A5} for the local solvability result on $L^{\infty}_{\sigma}$. For exterior domains solutions near the boundary may be understood as finite energy because of the compact boundary. On the other hand, for a half space global existence of non-decaying solutions is unknown. Theorem 1.1 implies existence of global solutions at least for slowly decaying data $u_0\in L^{\infty}_{\sigma}\cap \dot{H}^{1}_{0}(\mathbb{R}^{2}_{+})$. We refer to a recent paper \cite{GWittwer} on stationary solutions in a half space.

\vspace{20pt}

\section{The Stokes semigroup on $L^{\infty}\cap \dot{H}^{1}$}

\vspace{10pt}

In this section, we prove that the Stokes semigroup is an analytic semigroup on $L^{\infty}_{\sigma}\cap \h_0$. We first prove the a priori estimate of the Stokes flow (1.3) for compactly supported initial data by using fractional powers of the Stokes operator on $L^{2}$, and then extend initial data by an  approximation by using the Bogovski{\u\i} operator. After the proof of Theorem 2.1, we remark on higher dimensional cases (Remarks 2.6). \\

\vspace{5pt}

\begin{thm}
Let $\Omega$ be an exterior domain with $C^{3}$-boundary in $\mathbb{R}^{2}$. For $T_0>0$, there exists a constant $C$ such that the estimate 

\begin{align*}
\sup_{0<t\leq T_0}\Big\{\|S(t)v_0\|_{L^{\infty}\cap \h}
+t||A S(t)v_0||_{L^{\infty}\cap \h}\Big\}
\leq C||v_0||_{L^{\infty}\cap \h}  \tag{2.1}
\end{align*}\\
holds for $v_0\in L^{\infty}_{\sigma}\cap \H$. In particular, the Stokes semigroup forms a (not strongly continuous) analytic semigroup on $L^{\infty}_{\sigma}\cap \H$.
\end{thm}

\vspace{5pt}

In the sequel, we use the following regularity properties of the Stokes flow deduced from Theorem 2.1.

\vspace{5pt}

\begin{cor}
For $v_0\in L^{\infty}_{\sigma}\cap \h_0$, the Stokes flow $v=S(t)v_0$  satisfies 

\begin{align*}
&\sup_{0<t\leq T_0}\Big\{\|S(t)v_0\|_{L^{\infty}\cap \h}
+t^{\frac{1}{2}}||\nabla S(t)v_0||_{L^{\infty}}\Big\}
\leq C||v_0||_{L^{\infty}\cap \h},   \tag{2.2} \\[5pt]
&v, t^{\frac{1}{2}}\nabla v\in C_{w}([0,T_0]; L^{\infty}),\quad \nabla v\in C_{w}([0,T_0]; L^{2}).  \tag{2.3}
\end{align*}\\
\end{cor}

\begin{proof}
The estimate (2.2) follows from (2.1) and the $L^{\infty}$-estimate of the Stokes semigroup \cite{AG2}. Since $v=S(t)v_0$ converges to $v_0$ weakly-star on $L^{\infty}$ as $t\to0$, the function $t^{1/2}\nabla v$ converges to zero weakly-star on $L^{\infty}$. Moreover, $\nabla v$ converges to $\nabla v$ weakly on $L^{2}$ by (2.2).
\end{proof}

\vspace{5pt}

Let $H^{1}_{0,\sigma}$ denote the $H^{1}$-closure of $C_{c,\sigma}^{\infty}$, the space of all smooth solenoidal vector fields with compact support in $\Omega$. In oder to prove Theorem 2.1, we prepare a characterization of the space $H^{1}_{0,\sigma}$ and an approximation for $v_0\in L^{\infty}_{\sigma}\cap \h_{0}$ by compactly supported functions.

\vspace{10pt}

\begin{prop}
Let $\Omega$ be an exterior domain with Lipschitz boundary in $\mathbb{R}^{n}$, $n\geq 2$. Then, the space $H^{1}_{0,\sigma}$ agrees with the space of all solenoidal vector fields in $H^{1}_{0}$, where $H^{1}_{0}$ denotes the space of all functions in $H^{1}$, vanishing on $\partial\Omega$.
\end{prop}

\vspace{10pt}

\begin{lem}[Approximation]
Let $\Omega$ be an exterior domain in $\mathbb{R}^{2}$. There exists a constant $C$ such that for $v_0\in L^{\infty}_{\sigma}\cap \h_{0}$ there exists a sequence $\{v_{0,m}\}_{m=1}^{\infty} \subset L^{\infty}_{\sigma}\cap H^{1}_{0}$ with compact support in $\overline{\Omega}$ such that 

\begin{equation*}
\begin{aligned}
&||v_{0,m}||_{L^{\infty}\cap \h}\leq C||v_0||_{L^{\infty}\cap \h},\\
&v_{0,m}\to v_{0}\quad \textrm{a.e. in}\ \Omega \quad \textrm{as}\ m\to\infty.
\end{aligned}
\tag{2.4}
\end{equation*}
\end{lem}

\vspace{5pt}

We shall give proofs for Proposition 2.3 and Lemma 2.4 later and first complete:

\vspace{5pt}

\begin{proof}[Proof of Theorem 2.1]
Since the estimate

\begin{align*}
\sup_{0<t\leq T_0}\Big\{\|S(t)v_0\|_{L^{\infty}}
+t||A S(t)v_0||_{L^{\infty}}\Big\}
\leq C||v_0||_{L^{\infty}}  \tag{2.5}
\end{align*}\\
holds for $v_0\in L^{\infty}_{\sigma}\cap \h_{0}$ \cite{AG2}, it suffices to show 

\begin{align*}
\sup_{t>0}\Big\{ |S(t)v_0|_{ \h}
+t|A S(t)v_0|_{\h}\Big\}
\leq C||v_0||_{L^{\infty}\cap \h}.  \tag{2.6}
\end{align*}\\
We begin with $v_0\in L^{\infty}_{\sigma}\cap H^{1}_{0}$ with compact support in $\overline{\Omega}$. Since the Stokes operator $A$ is a  positive self-adjoint operator on $L^{2}$, we are able to define the fractional power $A^{1/2}$ by the spectral representation and we have 

\begin{align*}
&|u|_{\h}=||A^{\frac{1}{2}}u||_{L^{2}}\quad \textrm{for}\ u\in D(A^{\frac{1}{2}}),
\end{align*}\\
and $D(A^{1/2})=H^{1}_{0,\sigma}$ \cite[Lemma 2.2.1]{Sohr}. Since $v_0$ is supported in $\overline{\Omega}$ and vanishing on $\partial\Omega$, the initial data $v_0$ is an element of $H^{1}_{0,\sigma}$ by Proposition 2.3. It follows that 

\begin{align*}
|S(t)v_0|_{\h}=||A^{\frac{1}{2}}S(t)v_0||_{L^{2}}
&=||S(t)A^{\frac{1}{2}}v_0||_{L^{2}}\\
&\leq ||A^{\frac{1}{2}}v_0||_{L^{2}}=|v_0|_{\h}.
\end{align*}\\
Similarly, by the analyticity of the Stokes semigroup on $L^{2}$ \cite[IV 1.5]{Sohr}, we estimate 

\begin{align*}
|AS(t)v_0|_{\h}=||A^{\frac{3}{2}}S(t)v_0||_{L^{2}}
&=||AS(t)A^{\frac{1}{2}}v_0||_{L^{2}}\\
&\leq \frac{1}{t}||A^{\frac{1}{2}}v_0||_{L^{2}}
=\frac{1}{t}|v_0|_{\h}.
\end{align*}\\
Thus (2.6) holds.

For general $v_0\in L^{\infty}_{\sigma}\cap \h_{0}$, we apply Lemma 2.4 and take a sequence $\{v_{0,m}\} $ with compact support in $\overline{\Omega}$ satisfying (2.4). By applying (2.6) for $v_m=S(t)v_{0,m}$, we have  

\begin{align*}
\sup_{t>0}\Big\{ |v_m |_{ \h}
+t|A v_m|_{\h}\Big\}
\leq C||v_{0}||_{L^{\infty}\cap \h}.  
\end{align*}\\
Since $v_{0,m}\to v_{0}$ a.e. in $\Omega$, by choosing a subsequence, $v_{m}$ converges to $v=S(t)v_0$ locally uniformly in $\overline{\Omega}\times (0,\infty)$ \cite{AG2}. Thus the estimate (2.6) holds for $v_{0}\in L^{\infty}_{\sigma}\cap \h_{0}$.
\end{proof}

\vspace{10pt}

In the sequel, we prove Proposition 2.3 and Lemma 2.4 by using the Bogovski{\u\i} operator \cite{B79} (\cite[Theorem III 3.1]{Gal}). Let $L^{p}_{\textrm{av}}(D)$ denote the space of all average zero functions in $L^{p}(D)$, i.e., $\int_{D}f\dd x=0$. Let $W^{1,p}_{0}(D)$ denote the space of all functions in $W^{1,p}(D)$ vanishing on $\partial D$ for $p\in [1,\infty]$. 

\vspace{10pt}

\begin{prop}
Let $D$ be a bounded domain with Lipschitz boundary in $\mathbb{R}^{n}$, $n\geq 2$. Let $p\in (1,\infty)$. There exists a bounded operator $B: L^{p}_{\textrm{av}}(D)\to W^{1,p}_{0}(D)$ such that $u=B[g]$ satisfies $\D\ u=g$ in $D$, $u=0$ on $\partial D$ and 

\begin{align*}
||\nabla u||_{L^{p}(D)}\leq C||g||_{L^{p}(D)}.    \tag{2.7}
\end{align*}\\
The constant $C=C_{D}$ is independent of dilation of $D$, i.e., $C_{D}=C_{\lambda D}$ for $\lambda>0$.
\end{prop}

\vspace{5pt}

\begin{proof}[Proof of Proposition 2.3]
We reduce the problem to the cases of a bounded domain and the whole space by a cut-off function argument. We may assume $0\in \Omega^{c}$ by translation. Let $\theta\in C_{c}^{\infty}[0,\infty)$ be a cut-off function such that $\theta\equiv 1$ in $[0,1]$, $\theta\equiv 0$ in $[2,\infty)$ and $0\leq \theta\leq 1$. We set $\theta_R(x)=\theta(|x|/R)$ so that $\theta_R\equiv 1$ in $B_{0}(R)$ and $\theta_R\equiv 0$ in $\overline{B_{0}(2R)}^{c}$. Here, $B_{0}(R)$ denotes the open ball centered at the origin with radius $R$. We fix $R\geq 1$ sufficiently large so that $\Omega^{c}\subset B_{0}(R)$. For $v_0\in H^{1}_{0}$ satisfying $\D\ v_0=0$ in $\Omega$, we shall construct a sequence $\{v_{0,m}\}\subset C_{c,\sigma}^{\infty}$ such that $v_{0,m}\to v_0$ in $H^{1}$. Let $B=B_{D_R}$ be the Bogovski{\u\i} operator on $D_R=\{x\in \Omega\ |\ R<|x|<2R \}$. We decompose $v_0$ into two terms by setting 

\begin{align*}
v_{1}&=v_{0}\theta_{R}-B[g_{R}],\\
v_{2}&=v_{0}(1-\theta_{R})+B[g_{R}],
\end{align*}\\
for $g_{R}=v_0\cdot \nabla \theta_{R}$. Since the average of $g_{R}$ is zero in $D_{R}$ by $\D\ v_0=0$ in $\Omega$ and $v_{0}=0$ on $\partial \Omega$, $v_1$ and $v_2$ are defined by the Bogovski{\u\i} operator on $D_{R}$. Observe that $v_1\in H^{1}_{0}(\Omega_1)$ satisfies $\D\ v_1=0$ in $\Omega_1=\Omega\cap \{|x|<2R\}$, and $v_2\in H^{1}(\mathbb{R}^{n})$ satisfies $\D\ v_2=0$ in $\mathbb{R}^{n}$, $\textrm{spt}\ v_2\subset \{|x|\geq R\}$. Since the assertion of Proposition 2.3 is valid for bounded domains and the whole space \cite[II. 2.2.3 Lemma and II. 2.5.5 Lemma]{Sohr}, there exist sequences $\{v_{1,m}\}\subset C_{c,\sigma}^{\infty}(\Omega_1)$ and $\{v_{2,m}\}\subset C_{c,\sigma}^{\infty}(\mathbb{R}^{n})$ such that 

\begin{align*}
v_{1,m}&\to v_{1}\quad \textrm{in}\ H^{1}_{0}(\Omega_{1}),\\
v_{2,m}&\to v_{2}\quad \textrm{in}\ H^{1}(\mathbb{R}^{n}).
\end{align*}\\
We identify $v_{1,m}$ as an element of $C_{c,\sigma}^{\infty}(\Omega)$ by the zero extension to $\mathbb{R}^{n}\backslash \overline{\Omega_1}$. Since $v_{2}$ is supported in $\{|x|\geq R\}$, we may assume that $v_{2,m}$ is supported in $\Omega$. Since $v_{0}=v_{1}+v_{2}$, we obtain the desired sequence $v_{0,m}=v_{1,m}+v_{2,m}$ in $C_{c,\sigma}^{\infty}(\Omega)$. 
\end{proof}

\vspace{5pt}

It remains to show Lemma 2.4. Since the $\h$-semi-norm is invariant under the scaling $f_{\lambda}(x)= f(\lambda x)$, $\lambda>0$, for the two-dimensional space, the approximation (2.4) holds with the $L^{\infty}\cap \h$-norm (see Remarks 2.6 (iv) for $n\geq 3$).

\vspace{5pt}

\begin{proof}[Proof of Lemma 2.4]
Let $\theta$ be a cut-off function used in the proof of Proposition 2.3. We set $\theta_m(x)=\theta(|x|/m)$ so that $\theta_m\equiv 1$ in $B_{0}(m)$ and $\theta_m\equiv 0$ in $\overline{B_{0}(2m)}^{c}$, and take $m\geq 1$ sufficiently large so that $\Omega^{c}\subset B_{0}(m)$. For $v_0\in L^{\infty}_{\sigma}\cap \h_{0}$, we set 

\begin{align*}
v_{0,m}&=v_{0}\theta_m-u_m,\\
u_m&=B_{D_m}[g_m]\quad \textrm{for}\ g_m=v_0\cdot \nabla \theta_m,
\end{align*}\\
by the Bogovski{\u\i} operator $B_{D_m}$ on $D_m=\{x\in \Omega\ |\ m<|x|<2m \}$. We identify $u_{m}$ and its zero extension to $\mathbb{R}^{2}\backslash \overline{D_m}$. We observe that $v_{0,m}$ is with compact support in $\overline{\Omega}$ and satisfies $\D\ v_{0,m}=0$ in $\Omega$, $v_{0,m}=0$ on $\partial\Omega$. Since $v_{0,m}\to v_{0}$ a.e. in $\Omega$ as $m\to\infty$, it suffices to show 

\begin{align*}
||v_{0,m}||_{L^{\infty}\cap \dot{H}^{1}}\leq C||v_{0}||_{L^{\infty}\cap \dot{H}^{1}}.
\end{align*}\\
We observe that 

\begin{align*}
||v_0\theta_m||_{L^{\infty}}&\leq ||v_0||_{L^{\infty}},\\
||\nabla (v_0\theta_m)||_{L^{2}}&\leq ||\nabla v_0||_{L^{2}}+||v_0||_{L^{\infty}}||\nabla \theta_m||_{L^{2}}.\\
\end{align*}\\
Since the $\h$-semi-norm is invariant for $\theta_{m}(x)= \theta(|x|/m)$, i.e., 

\begin{align*}
||\nabla \theta_m||_{L^{2}(\mathbb{R}^{2})}=||\nabla \theta||_{L^{2}(\mathbb{R}^{2})},
\end{align*}\\
we have 

\begin{align*}
||v_0\theta_m||_{L^{\infty}\cap \h}\leq C ||v_0||_{L^{\infty}\cap \h},
\end{align*}\\
with some constant $C$, independent of $m\geq 1$. We shall show 

\begin{align*}
||u_m||_{L^{\infty}(\Omega)}+||\nabla u_m||_{L^{2}(\Omega)}
\leq C||v_0||_{L^{\infty}(\Omega)}.   \tag{2.8}
\end{align*}\\
The desired estimate for $v_{0,m}$ follows from (2.8). We estimate the cut-off function $\theta_m$ and observe that 

\begin{align*}
||g_m||_{L^{p}(D_m)}\leq \frac{C_1}{m^{1-\frac{2}{p}}}||v_0||_{L^{\infty}(\Omega)}
\end{align*}\\
holds for $p\in [1, \infty]$. We apply (2.7) to estimate  

\begin{align*}
||\nabla u_m||_{L^{p}(D_m)}\leq C_2||g_m||_{L^{p}(D_m)}.  \tag{2.9}
\end{align*}\\
Since the constant in (2.7) is invariant under the dilation, the constant $C_{2}$ is independent of $m$. We take $p=2$ and obtain a uniform estimate for $\nabla u_m$ in $L^{2}$. It remains to show the $L^{\infty}$-estimate for $u_m$. Since $u_m$ vanishes on $\partial D_m$, we apply the Poincar\'e inequality \cite{E} to estimate 

\begin{align*}
||u_m||_{L^{p}(D_m)}\leq mC_3||\nabla u_m||_{L^{p}(D_m)}.
\end{align*}\\
with the constant $C_3$, independent of $m$. The above estimate is obtained by applying the inequality in $D_{1}$ and rescaling. We apply the Sobolev inequality \cite[Lemma 3.1.4]{Lunardi} for $p>2$ to estimate 

\begin{align*}
||u_m||_{L^{\infty}(D_m)}
&\leq C_{4}||u_m||_{L^{p}(D_m)}^{1-\frac{2}{p}}||\nabla u_m||_{L^{p}(D_m)}^{\frac{2}{p}}\\
&\leq C_{4}{(mC_{3})}^{1-\frac{2}{p}}||\nabla u_m||_{L^{p}(D_m)}\\
&\leq C_{4} {C_{3}}^{1-\frac{2}{p}}C_2C_1||v_0||_{L^{\infty}(\Omega)}.
\end{align*}\\
Since the constants $C_1-C_4$ are independent of $m$, we obtain the estimate (2.8). The proof is now complete.
\end{proof}

\vspace{10pt}

\begin{rems}

\noindent
(i) When $\Omega$ is the whole space, the heat semigroup forms a bounded analytic semigroup on $L^{\infty}\cap \dot{W}^{1,n}$, where $\dot{W}^{1,n}$ denotes the homogeneous $L^{n}$-Sobolev space. We set 

\begin{align*}
&|f|_{\dot{W}^{1,n}}=||\nabla f||_{L^{n}},\\
&||f||_{L^{\infty}\cap \dot{W}^{1,n}}=||f||_{L^{\infty}}+|f|_{\dot{W}^{1,n}}.
\end{align*}\\
The space $L^{\infty}\cap \dot{W}^{1,n}$ is a Banach space equipped with the norm $||\cdot ||_{L^{\infty}\cap \dot{W}^{1,n}}$. Since spatial derivatives commute with the semigroup, we estimate 

\begin{align*}
&|e^{t\Delta}v_0|_{\dot{W}^{1,n}}
=|| e^{t\Delta}\nabla v_0||_{L^{n}}
\leq C||\nabla v_0||_{L^{n}},\\
&|\Delta e^{t\Delta}v_0|_{\dot{W}^{1,n}}
=||\Delta e^{t\Delta}\nabla v_0||_{L^{n}}
\leq \frac{C'}{t}||\nabla v_0||_{L^{n}},\quad t>0.
\end{align*}\\
Since the heat semigroup is a $C_0$-semigroup on $L^{n}$, it is strongly continuous for $v_0\in L^{\infty}\cap \dot{W}^{1,n}$, i.e., $\nabla e^{t\Delta}v_0\to \nabla v_0$ in $L^{n}$ as $t\to0$.

\noindent 
(ii) When $\Omega$ is a half space, the Stokes semigroup is a bounded analytic semigroup on $L^{\infty}_{\sigma}\cap \dot{W}^{1,n}_{0}$, where $\dot{W}^{1,n}_{0}$ denotes the space of all functions in $\dot{W}^{1,n}$, vanishing on the boundary. We estimate the $\dot{W}^{1,n}$-semi-norm by the estimates of the fractional power 
 
\begin{align*}
||\nabla u||_{L^{n}}&\leq C_1||A^{\frac{1}{2}}u||_{L^{n}},   \tag{2.10}\\
||A^{\frac{1}{2}}u||_{L^{n}}&\leq C_2||\nabla u||_{L^{n}}, \tag{2.11}
\end{align*} \\
for $u\in D(A^{1/2})$ and $D(A^{1/2})=W^{1,n}_{0,\sigma}$, where $W^{1,n}_{0,\sigma}$ denotes the $W^{1,n}$-closure of $C_{c,\sigma}^{\infty}$. The estimate (2.10) is proved in \cite[Theorem 3.6 (ii)]{BM88} for general $L^{p}$-norms. The estimate (2.11) follows from a duality; see Appendix B. Applying (2.10) and (2.11) together with the analyticity of $S(t)$ on $L^{n}$ implies that 

\begin{align*}
|S(t)v_0|_{\dot{W}^{1,n}}
=||\nabla S(t)v_0||_{L^{n}}
&\leq C_1||A^{\frac{1}{2}}S(t)v_0||_{L^{n}}\\
&\leq C_1'||A^{\frac{1}{2}}v_0||_{L^{n}}\\
&\leq C||\nabla v_0||_{L^{n}},\\
|AS(t)v_0|_{\dot{W}^{1,n}}
&\leq \frac{C'}{t}||\nabla v_0||_{L^{n}},\quad t>0,
\end{align*}\\
for $v_0\in D(A^{1/2})$. Since we are able to approximate $v_0\in L^{\infty}_{\sigma}\cap \dot{W}^{1,n}_{0}$ by elements of $D(A^{1/2})$ by a pointwise convergence as remarked below in (iv), we have 

\begin{align*}
|S(t)v_0|_{\dot{W}^{1,n}}+t|AS(t)v_0|_{\dot{W}^{1,n}}
\leq C||v_0||_{L^{\infty}\cap \dot{W}^{1,n}},\quad t>0
\end{align*}\\
for $v_0\in L^{\infty}_{\sigma}\cap \dot{W}^{1,n}_{0}$. Since the Stokes semigroup is a bounded analytic semigroup on $L^{\infty}_{\sigma}$ \cite{DHP}, \cite{Sl03}, it forms a bounded analytic semigroup on $L^{\infty}_{\sigma}\cap \dot{W}^{1,n}_{0}$ for $n\geq 2$. 

\noindent 
(iii) When $\Omega$ is bounded, $S(t)$ is a bounded analytic semigroup on $L^{\infty}_{\sigma}$ \cite{AG1}. Moreover, the estimates (2.10) and (2.11) hold (see Remark B.2). Thus $S(t)$ is regarded as a bounded analytic semigroup on $L^{\infty}_{\sigma}\cap \dot{W}^{1,n}_{0}$. When $\Omega$ is an exterior domain, the estimate (2.10) does not hold for $n\geq 3$ \cite[Theorem 1.1 (ii)]{BM1992} (see also \cite[Theorem B]{GS}, \cite[Theorem 4.4]{BM}.) When $n=2$, $L^{2}$-theory is available and we obtained (2.1).

\noindent 
(iv) The assertion of Lemma 2.4 is extendable for higher dimensional cases for $n\geq 3$. Since the $\dot{W}^{1,n}$-semi-norm is invariant under the scaling $f_{\lambda}(x)=f(\lambda x)$ for $\lambda>0, x\in \mathbb{R}^{n}$, by the same way as we proved Lemma 2.4, for $v_0\in L^{\infty}_{\sigma}\cap \dot{W}^{1,n}_{0}$ we are able to construct a sequence $\{v_{0,m}\}\subset L^{\infty}_{\sigma}\cap {W}^{1,n}_{0}$ with compact support in $\overline{\Omega}$ satisfying 

\begin{equation*}
\begin{aligned}
&||v_{0,m}||_{L^{\infty}\cap \dot{W}^{1,n}}\leq C ||v_{0}||_{L^{\infty}\cap \dot{W}^{1,n}},\\
&v_{0,m}\to v_{0}\quad \textrm{a.e.}\ \textrm{in}\ \Omega,
\end{aligned}
\tag{2.12}
\end{equation*}\\
with some constant $C$. By a similar cut-off function argument, we are able to construct a sequence satisfying (2.12) also for a half space $\Omega=\mathbb{R}^{n}_{+}$, $n\geq 2$. 
\end{rems}

\vspace{10pt}

\section{Local solvability of the perturbed system}

\vspace{10pt}

In this section, we prove local solvability of the perturbed system (1.5) for non-decaying initial data in $L^{\infty}\cap \h$ (Theorem 3.1). The local solvability on $L^{\infty}\cap \h$ is used in the next section in order to construct global solutions for $w_0=0$. We apply regularizing estimates of the Stokes semigroup proved in the previous section and establish the solvability on $L^{\infty}\cap \dot{H}^{1}$ by an iterative argument.\\

\vspace{5pt}

We consider the perturbed system of the form

\begin{equation*}
\begin{aligned}
\partial_t w-\Delta{w}+w\cdot \nabla w+v\cdot \nabla w+\nabla{\pi}&= -w\cdot \nabla v-v\cdot \nabla v \quad \textrm{in}\quad \Omega\times(0,T_0),    \\
\D\ w&=0\quad\qquad\qquad\qquad \textrm{in}\quad \Omega\times(0,T_0),    \\
w&=0\qquad\qquad\qquad\quad \textrm{on}\quad \partial\Omega\times (0,T_0) ,  \\
w&=w_0\hspace{70pt}\textrm{on}\quad  \Omega\times \{t=0\}.  
\end{aligned}
\tag{3.1}
\end{equation*}\\
We assume that the perturbation $v$ is a solenoidal vector field in $\Omega$, vanishing on the boundary, i.e., $\D\ v=0$ in $\Omega$ and $v=0$ on $\partial\Omega$, and satisfies 
 
\begin{align*}
v,\ t^{1/2}\nabla v\in C_{w}([0,T_0]; L^{\infty}),\quad \nabla v\in C_{w}([0,T_0]; L^{2}).   \tag{3.2}
\end{align*}\\
Our goal is to prove local existence of mild solutions of the form

\begin{align*}
w(t)=S(t)w_0-\int_{0}^{t}S(t-s)\mathbb{P}(\tilde{w}\cdot \nabla \tilde{w})\dd s,\quad \tilde{w}=w+v.      
\end{align*}

\vspace{5pt}

\begin{thm}
Let $v$ be a solenoidal vector field in $\Omega$ satisfying (3.2). Set

\begin{align*}
N=\sup_{0< t\leq T_0}\Big\{||v||_{L^{\infty}\cap \h}(t)+t^{1/2}||\nabla v||_{L^{\infty}}(t)  \Big\}.
\end{align*}\\
There exists a constant $\varepsilon$ such that for $w_0\in L^{\infty}_{\sigma}\cap \H$, there exists $T\geq \varepsilon K^{-2}$ for $K=||w_0||_{L^{\infty}\cap \h}+N$, and a unique mild solution $w\in C_{w}([0,T]; L^{\infty})$ of (3.1) satisfying $t^{1/2}\nabla w\in C_{w}([0,T]; L^{\infty})$ and $\nabla w\in C_{w}([0,T]; L^{2})$.
\end{thm}

\vspace{10pt}

In order to prove Theorem 3.1, we recall $L^{\infty}$-estimates of the Stokes semigroup.

\vspace{10pt}

\begin{prop}
Let $T_0>0$ and $\alpha\in (0,1)$. There exist constants $C_1-C_4$ such that the estimates \\
\begin{align*}
||S(t)w_0||_{L^{\infty}\cap \h}
+t^{\frac{1}{2}}||\nabla S(t)w_0||_{L^{\infty}}
&\leq C_1||w_0||_{L^{\infty}\cap \h},   \tag{3.3}\\
t^{\frac{|k|}{2}}||\partial_x^{k}S(t)\mathbb{P} f||_{L^{2}}&\leq C_2 ||f||_{L^{2}}, \tag{3.4}\\
t^{\frac{|k|+1-\alpha}{2}}||\partial_x^{k}S(t)\mathbb{P} \D\ F||_{L^{\infty}}
&\leq C_3||F||_{L^{\infty}}^{1-\alpha}||F||_{W^{1,\infty}}^{\alpha},          \tag{3.5}\\
t^{\frac{1}{2}}||S(t)\mathbb{P} f||_{L^{\infty}\cap \dot{H}^{1}}&\leq C_4 ||f||_{L^{2}},    \tag{3.6}
\end{align*}\\
hold for $w_0\in L^{\infty}_{\sigma}\cap \H$, $f\in L^{2}$, $F\in W^{1,\infty}_{0}$ satisfying $\D\ F\in L^{2}$, $|k|\leq 2$, and $0<t\leq T_0$.
\end{prop}

\vspace{5pt}

\begin{proof}
The estimate (3.3) follows from Corollary 2.2. Since the Stokes semigroup $S(t)$ is an analytic semigroup on $L^{2}_{\sigma}$, the estimate (3.4) holds \cite[IV.1.5]{Sohr}. The estimate (3.5) is proved in \cite[Theorem 1.1]{A3} for $F\in C_{c}^{\infty}$. It is proved in \cite[Theorem 2.2]{A5} that the composition operator $S(t)\p\D$ is uniquely extendable to a bounded operator $\overline{S(t)\p\D}$ from $W^{1,\infty}_{0}$ to $L^{\infty}_{\sigma}$ together with (3.5). Since $\overline{S(t)\p\D} F$ agrees with $S(t)\p\D F $ for $F\in W^{1,\infty}_{0}$ satisfying $\D\ F\in L^{2}$, we have (3.5).

It remains to show (3.6). Since $t^{1/2}||\nabla S(t)\p f||_{2}\leq C||f||_{2}$ by (3.4), it suffices to show that 

\begin{align*}
||S(t)\p f||_{L^{\infty}}\leq \frac{C}{t^{\frac{1}{2}}}||f||_{L^{2}}.  \tag{3.7}
\end{align*}\\
We estimate $v=S(t)\p f$ by interpolation. By the Sobolev inequality \cite[Lemma 3.1.4]{Lunardi}, we have

\begin{align*}
||v||_{L^{\infty}}\leq C||v||_{L^{4}}^{\frac{1}{2}}||\nabla v||_{L^{4}}^{\frac{1}{2}}.  
\end{align*}\\
We estimate the $L^{4}$-norm by the $H^{1}$-norm. We invoke the Sobolev inequality

\begin{align*}
||\varphi||_{L^{4}}\leq C||\varphi||_{L^{2}}^{\frac{1}{2}}|| \varphi||_{H^{1}}^{\frac{1}{2}} \tag{3.8}
\end{align*}\\
for $\varphi\in H^{1}(\Omega)$. The estimate (3.8) is known for $\Omega=\mathbb{R}^{2}$ (e.g., \cite[Chapter 1, Lemma 2]{Lady61}) and extendable for the exterior domain $\Omega\subset \mathbb{R}^{2}$ by a suitable extension of $\varphi\in H^{1}(\Omega)$ to $\mathbb{R}^{2}\backslash \overline{\Omega}$, . We apply (3.8) for $\varphi=v$ and $\nabla v$. By $L^{2}$-estimates of the Stokes semigroup (3.4) for $|k|\leq 2$, we obtain (3.7). The proof is complete.
\end{proof}

\vspace{5pt}

\begin{proof}[Proof of Theorem 3.1]
We set a sequence $\{w_j\}$ by

\begin{equation*}
\begin{aligned}
w_{j+1}&=S(t)w_0-\int_{0}^{t}S(t-s)\mathbb{P}\tilde{w}_{j}\cdot \nabla \tilde{w}_{j}\dd s, \quad  \tilde{w}_{j}=w_{j}+v,\\
w_1&=S(t)w_0,
\end{aligned}
\tag{3.9}
\end{equation*}\\
and the constants 

\begin{align*}
K_{j}&=\sup_{0<t\leq T}\Big\{ ||w_j||_{L^{\infty}\cap \h}(t)+t^{1/2}||\nabla w_{j}||_{L^{\infty}}(t)   \Big\},\\
\tilde{K}_{j}&=\sup_{0<t\leq T}\Big\{ ||\tilde{w}_j||_{L^{\infty}\cap \h}(t)+t^{1/2}||\nabla \tilde{w}_{j}||_{L^{\infty}}(t)   \Big\}.
\end{align*}\\
We may assume that $T\leq 1$. We first show that
 
\begin{align*}
K_{j}\leq CK\quad \textrm{for}\ j=1,2,\cdots,  \tag{3.10}
\end{align*}\\
for $T\leq \varepsilon_0 K^{-2}$ with some constant $\varepsilon_0>0$ and $K=||w_0||_{L^{\infty}\cap \h}+N$. We apply (3.3) to estimate 

\begin{align*}
||w_{j+1}||_{L^{\infty}\cap \h}\leq C_1||w_0||_{L^{\infty}\cap \h}
+\int_{0}^{t}||S(t-s)\mathbb{P}\tilde{w}_{j}\cdot \nabla \tilde{w}_{j}||_{L^{\infty}\cap \h}\dd s. 
\end{align*}\\
We set $\tilde{F}_{j}=\tilde{w}_{j}\tilde{w}_{j}$. Since $\D\ \tilde{F}_{j}=\tilde{w}_{j}\cdot \nabla \tilde{w}_{j}$, it follows that 

\begin{align*}
||\tilde{F}_{j}||_{\infty}&\leq \tilde{K}_j^{2}, \\
||\nabla \tilde{F}_{j}||_{\infty}&\leq \frac{2}{s^{1/2}}\tilde{K}_j^{2},\\
||\D\ \tilde{F}_j||_{2}&\leq \tilde{K}_j^{2}.
\end{align*}\\
Since $s\leq  1$ for $t\leq T$, we observe that 

\begin{align*}
||\tilde{F}_{j}||_{\infty}^{1-\alpha}||\tilde{F}_{j}||_{1,\infty}^{\alpha}
&\leq \tilde{K}_{j}^{2(1-\alpha)}\Bigg(\tilde{K}_{j}^{2}+\frac{2}{s^{\frac{1}{2}}}\tilde{K}_{j}^{2}  \Bigg)^{\alpha}\\
&=\Bigg(1+\frac{2}{s^{\frac{1}{2}}}\Bigg)^{\alpha}\tilde{K}_{j}^{2}\\
&\leq \frac{3^{\alpha}}{s^{\frac{\alpha}{2}}}\tilde{K}_{j}^{2},
\end{align*}\\
where $||\tilde{F}_{j}||_{1,\infty}=||\tilde{F}_{j}||_{W^{1,\infty}}$. We apply (3.5) and estimate 

\begin{align*}
\int_{0}^{t}||S(t-s)\p \D\ \tilde{F}_{j}||_{\infty}\dd s
&\leq \int_{0}^{t}\frac{C_3}{(t-s)^{\frac{1-\alpha}{2}}}||\tilde{F}_{j}||_{\infty}^{1-\alpha}|| \tilde{F}_{j}||_{1,\infty}^{\alpha}\dd s \\
&\leq 3^{\alpha}C_3 \tilde{K}_{j}^{2}\int_{0}^{t}\frac{\dd s}{(t-s)^{\frac{1-\alpha}{2}}s^{\frac{\alpha}{2}}   } \\
&=3^{\alpha}C_3M_1 \tilde{K}_{j}^{2} t^{\frac{1}{2}},
\end{align*}
where

\begin{align*}
M_{l}=\int_{0}^{1}\frac{\dd \eta}{(1-\eta)^{\frac{l-\alpha}{2}}{\eta}^{\frac{\alpha}{2}}}  \quad \textrm{for}\ l=1,2.
\end{align*}\\
By (3.4) we estimate 

\begin{align*}
\int_{0}^{t}||\nabla S(t-s)\p \D\ \tilde{F}_{j}||_{2}\dd s
&\leq \int_{0}^{t}\frac{C_2}{(t-s)^{\frac{1}{2}}}||\D\ \tilde{F}_{j}||_{2}\dd s \\
&\leq 2C_2 \tilde{K}_j^{2}t^{\frac{1}{2}}. 
\end{align*}\\
We obtain 

\begin{align*}
||w_{j+1}||_{L^{\infty}\cap \h}\leq C_1||w_0||_{L^{\infty}\cap \h}
+(3^{\alpha}C_3M_1+2C_2)\tilde{K}_j^{2}t^{\frac{1}{2}}.
\end{align*}\\
Similarly, applying (3.3) and (3.5) implies 

\begin{align*}
t^{\frac{1}{2}}||\nabla w_{j+1}||_{L^{\infty}}\leq C_1||w_0||_{L^{\infty}\cap \h}+3^{\alpha}C_3M_2\tilde{K}_{j}^{2}t^{\frac{1}{2}}.
\end{align*}\\
By combining the above estimates and 

\begin{align*}
\tilde{K}_{j}
&\leq K_j+N\\
&\leq K_j+K, 
\end{align*}\\
we obtain 

\begin{align*}
K_{j+1}\leq C_5 K+C_6(K_j+K)^{2}T^{\frac{1}{2}}
\end{align*} \\
with the constants 

\begin{align*}
C_5&=2C_1,\\
C_6&=2^{\alpha}C_3(M_1+M_2)+2C_2.
\end{align*}\\
We take $T\leq \varepsilon_0 K^{-2}$ for
 
\begin{align*}
\varepsilon_0=\Bigg(\frac{C_5}{C_6(2C_5+1)^{2}}\Bigg)^{2},
\end{align*}\\
and obtain (3.10) for $C=2C_5$. Since $S(t)w_0$ is weakly-star continuous on $L^{\infty}$ at time zero by Corollary 2.2, the sequence $\{w_j\}$ is uniformly bounded in $C_{w}([0,T]; L^{\infty})$. Moreover, the sequence $w_j$ satisfies $t^{1/2}\nabla w_j\in C_{w}([0,T]; L^{\infty})$ and $\nabla w_j \in C_{w}([0,T]; L^{2})$.\\

We show that $\{w_j\}$ converges to a limit $w\in C_{w}([0,T]; L^{\infty})$. We set 

\begin{align*}
&\rho_j=w_{j}-w_{j-1},\\
&L_{j}=\sup_{0<t\leq T}\Big\{||\rho_j||_{L^{\infty}\cap \h}(t)+t^{\frac{1}{2}}||\nabla \rho_j||_{L^{\infty}}(t)  \Big\}.
\end{align*}\\
Since 

\begin{align*}
\tilde{w}_{j}\cdot \nabla \tilde{w}_{j}-\tilde{w}_{j-1}\cdot \nabla\tilde{w}_{j-1}
&=\tilde{w}_{j}\cdot \nabla (\tilde{w}_{j}-\tilde{w}_{j-1})+\tilde{w}_{j}\cdot \nabla\tilde{w}_{j-1}
-(\tilde{w}_{j-1}-\tilde{w}_{j})\cdot \nabla\tilde{w}_{j-1}
-\tilde{w}_{j}\cdot \nabla\tilde{w}_{j-1}\\
&=\tilde{w}_{j}\cdot \nabla\rho_j+\rho_j\cdot \nabla \tilde{w}_{j-1},
\end{align*}\\
we estimate 

\begin{align*}
\rho_{j+1}=-\int_{0}^{t}S(t-s)\p(\tilde{w}_{j}\cdot \nabla\rho_j+\rho_j\cdot \nabla \tilde{w}_{j-1})\dd s.
\end{align*}\\
We set $F_1=\tilde{w}_{j}\rho_j$. Since $\D\ F_1=\tilde{w}_{j}\cdot \nabla \rho_j$, we have 

\begin{align*}
||{F}_{1}||_{\infty}&\leq \tilde{K}_jL_j, \\
||\nabla {F}_{1}||_{\infty}&\leq \frac{2}{s^{1/2}}\tilde{K}_jL_j,\\
||\D\ {F}_{1}||_{2}&\leq \tilde{K}_jL_j.
\end{align*}\\
Since $s\leq 1$, it follows that 

\begin{align*}
||F_1||_{\infty}^{1-\alpha}|| F_1||_{1,\infty}^{\alpha}
&\leq (\tilde{K}_j L_j)^{1-\alpha}\Bigg(\tilde{K}_j L_j+\frac{2}{s^{\frac{1}{2}}}\tilde{K}_j L_j\Bigg)^{\alpha}\\
&\leq \frac{3^{\alpha}}{s^{\frac{\alpha}{2}}}\tilde{K}_j L_j.
\end{align*}\\
We apply (3.5) to estimate 

\begin{align*}
\int_{0}^{t}||S(t-s)\p\D\ F_1||_{\infty}\dd s
&\leq \int_{0}^{t}\frac{C_3}{(t-s)^{\frac{1-\alpha}{2}}}||F_1||_{\infty}^{1-\alpha}|| F_1||_{1,\infty}^{\alpha}\dd s\\
&\leq 3^{\alpha} C_3 \tilde{K}_{j}L_j\int_{0}^{t}\frac{\dd s}{(t-s)^{\frac{1-\alpha}{2}}s^{\frac{\alpha}{2}}}\\
&=3^{\alpha} C_3M_1 \tilde{K}_{j}L_j  t^{\frac{1}{2}}.
\end{align*}\\
Similarly, applying (3.5) and (3.4) yields 

\begin{align*}
\int_{0}^{t}||\nabla S(t-s)\p\D\ F_1||_{\infty}\dd s
&\leq 3^{\alpha}C_3 M_2 \tilde{K}_{j}L_{j},\\
\int_{0}^{t}||\nabla S(t-s)\p\D\ F_1||_{2}\dd s
&\leq 2C_2 \tilde{K}_{j}L_{j}t^{\frac{1}{2}}.
\end{align*}\\ 
By combining the above estimates, we have  

\begin{align*}
&\int_{0}^{t}||S(t-s)\p\D\ F_1||_{L^{\infty}\cap \h}\dd s
+t^{\frac{1}{2}}\int_{0}^{t}||\nabla S(t-s)\p\D\ F_1||_{L^{\infty}}\dd s \\
&\leq (3^{\alpha}C_3(M_1+M_2)+2C_2 )\tilde{K}_{j}L_{j}t^{\frac{1}{2}}.
\end{align*}\\
By a similar way, we set $F_2=\rho_j\tilde{w}_{j-1}$ and estimate

\begin{align*}
&\int_{0}^{t}||S(t-s)\p\D\ F_2||_{L^{\infty}\cap \h}\dd s
+t^{\frac{1}{2}}\int_{0}^{t}||\nabla S(t-s)\p\D\ F_2||_{2}\dd s \\
&\leq (3^{\alpha}C_3(M_1+M_2)+2C_2 )\tilde{K}_{j-1}L_{j}t^{\frac{1}{2}}.
\end{align*}\\
By combining the above estimates for $F_1$ and $F_2$, we obtain

\begin{align*}
L_{j+1}\leq C_7 (\tilde{K}_{j}+\tilde{K}_{j-1})L_jT^{\frac{1}{2}}
\end{align*}\\
with the constant $C_7=3^{\alpha}C_3 (M_1+M_2)+2C_2$. Since the estimate (3.10) holds for $C=2C_5$, it follows that 

\begin{align*}
\tilde{K}_{j}
&\leq K_j+N\\
&\leq 2C_5 K+N\\
&\leq(2C_5+1)K.
\end{align*}\\
We take $T\leq \varepsilon_1 K^{-2}$ for 

\begin{align*}
\varepsilon_1=\Bigg(\frac{1}{4C_7(2C_5+1)}\Bigg)^{2},
\end{align*}\\
so that 

\begin{align*}
L_{j+1}
\leq \frac{1}{2}L_{j}\quad \textrm{for}\ j=1,2,\cdots.
\end{align*}\\
It follows that

\begin{align*}
L_{j+1}
\leq \Big(\frac{1}{2}\Big)^{j}L_1\to 0\quad \textrm{as}\ j\to\infty.
\end{align*}\\
Thus, $w_{j}$ converges to a limit $w\in C_{w}([0,T]; L^{\infty})$ uniformly in $\overline{\Omega}\times [0,T]$ and the limit $w$ satisfies $t^{1/2}\nabla w\in C_{w}([0,T]; L^{\infty})$ and $\nabla w\in C_{w}([0,T]; L^{2})$. By sending $j\to\infty$ to (3.9), the limit $w$ satisfies the integral equation. It is not difficult to show the uniqueness by estimating difference of two mild solutions. We proved the existence of mild solution for $T>0$ satisfying $T\leq \varepsilon_2 K^{-2}$ for $\varepsilon=\min\{\varepsilon_0,\varepsilon_1 \}$. Thus the existence time is estimated from below by $T\geq \varepsilon K^{-2}$ for 

\begin{align*}
\varepsilon=\frac{\varepsilon_2}{2}.
\end{align*}\\
The proof is now complete.
\end{proof}

\vspace{5pt}

\vspace{10pt}

\section{Global solutions of the perturbed system}

\vspace{10pt}

The goal of this section is to construct global solutions of the perturbed system (3.1) for $w_0=0$ (Theorem 4.5). We first show that mild solutions on $L^{\infty}\cap \h$ satisfy the system (3.1) in a suitable sense and a global energy estimate holds. In the subsequent section, we prove that mild solutions are globally bounded on $L^{\infty}$ by using an $L^{p}$-blow-up estimate. We then estimate a global $H^{1}$-norm of mild solutions by using an integral form.

\vspace{5pt}

\subsection{Energy estimates}

We first show that the mild solution for $w_0=0$ satisfies (3.1) on $L^{2}$.

\vspace{5pt}

\begin{lem}
Assume that $w_0=0$. Then, the mild solution $w$ of (3.1) satisfies

\begin{equation*}
\begin{aligned}
w\in C([0,T]; L^{p}),\quad 2\leq p\leq \infty,\\
\nabla w\in C([0,T]; L^{q}),\quad 2\leq q<\infty.
\end{aligned}
\tag{4.1}
\end{equation*}\\
Moreover, 

\begin{align*}
\partial_t w,\ \nabla^{2}w\in L^{s}(0,T; L^{2}),\quad   s\in (1,\infty),
\tag{4.2}
\end{align*}\\
and $w$ satisfies (3.1) on $L^{2}$ for a.e. $t\in (0,T)$ with the associated pressure $\pi$.
\end{lem}

\vspace{5pt}

In order to investigate regularity properties of mild solutions, we prepare an $L^{p}$-estimate of the Stokes semigroup.

\vspace{5pt}

\begin{prop}
There exists a constant $C$ such that 

\begin{align*}
||\partial_x^{k}S(t)\p f||_{L^{p}}\leq \frac{C}{t^{\frac{|k|+1}{2}-\frac{1}{p}  }}||\p f||_{L^{2}}  \tag{4.3}
\end{align*}\\
holds for $f\in L^{2}$, $2\leq p\leq \infty$, $|k|\leq 1$ and $0<t\leq T_0$.
\end{prop}

\vspace{5pt}

\begin{proof}
It suffices to show the cases $p=2$ and $p=\infty$. Then, the desired estimates follow from interpolation. When $k=0$, (4.3) holds for $p=2$ and $p=\infty$ by (3.4) and (3.6). When $|k|=1$, (4.3) holds for $p=2$ by (3.4). By the $L^{\infty}$-estimate of the Stokes semigroup \cite{AG2}, we have 

\begin{align*}
||\nabla S(t)\p f||_{\infty}
&=\Big\|\nabla S\Big(\frac{t}{2}\Big)S\Big(\frac{t}{2}\Big)\p f\Big\|_{\infty}\\
&\leq \frac{C}{t^{\frac{1}{2}}}\Big\|S\Big(\frac{t}{2}\Big)\p f\Big\|_{\infty}.
\end{align*}\\
Since (4.3) holds for $|k|=0$ and $p=\infty$, we obtain 

\begin{align*}
||\nabla S(t)\p f||_{\infty}\leq \frac{C}{t}||\p f||_{2}.
\end{align*}\\
We proved (4.3) for $|k|=1$ and $p=\infty$. 
\end{proof}

\vspace{5pt}

\begin{proof}[Proof of Lemma 4.1]
We estimate the mild solution

\begin{align*}
w=\int_{0}^{t}S(t-s)\p f\dd s
\end{align*}\\
for $f=-\tilde{w}\cdot\nabla \tilde{w}$ and $\tilde{w}=w+v$. Since $\tilde{w}$ satisfies $\tilde{w}\in C_{w}([0,T]; L^{\infty})$ and $\nabla \tilde{w}\in C_{w}([0,T]; L^{2})$ by Theorem 3.1 and (3.2), we observe that 

\begin{align*}
f\in L^{\infty}(0,T; L^{2}).
\end{align*}\\
If follows from (4.3) that 

\begin{align*}
||w||_{L^{p}}
&\leq \int_{0}^{t}||S(t-s)\mathbb{P}f||_{L^{p}}\dd s\\
&\leq C\int_{0}^{t}\frac{\dd s}{(t-s)^{\frac{1}{2}-\frac{1}{p} }}||\p f||_{L^{\infty}(0,T; L^{2})}\\
&\leq Ct^{\frac{1}{2}+\frac{1}{p}}||f||_{L^{\infty}(0,T; L^{2})}\to 0\quad \textrm{as}\ t\to0.
\end{align*}\\
Thus, $w\in C([0,T]; L^{p})$ for $p\in [2,\infty]$. By a similar way, applying (4.3) for $|k|=1$ implies that $\nabla w\in C([0,T]; L^{q})$ for $q\in [2,\infty)$.\\

It remains to show (4.2). By the maximal regularity estimate of the Stokes operator $A$ \cite[Theorem 4.4]{deSimon64} (\cite[IV. 1.6.2 Lemma]{Sohr}), we estimate 

\begin{align*}
||\partial_t w||_{L^{s}(0,T; L^{2})}+||A w||_{L^{s}(0,T; L^{2})}\leq C||f||_{L^{s}(0,T; L^{2})}\quad \textrm{for}\ s\in (1,\infty).
\end{align*}\\
By the resolvent estimate, we have 

\begin{align*}
||\nabla^{2}w||_{2}\leq C'(||w||_{2}+||Aw||_{2}).
\end{align*} \\
Thus (4.2) holds and $w$ satisfies (3.1) on $L^{2}$ with the associated pressure $\pi$. 
\end{proof}

\vspace{5pt}

Since the mild solution $w$ satisfies (3.1) on $L^{2}$, by integration by parts we obtain the global energy estimate.

\vspace{5pt}

\begin{lem}
The estimate  

\begin{align*}
\int_{\Omega}|w|^{2}\dd x+\int_{0}^{t}\int_{\Omega}|\nabla w|^{2}\dd x\dd s\leq \frac{1}{2}(e^{2N^{2}t}-1)\quad 0\leq t\leq T_0   \tag{4.4}
\end{align*}\\
holds for mild solutions of (3.1) for $w_0=0$ and $T_0>0$ with the constant $N$ in Theorem 3.1. In particular, $w\in L^{4}(0,T_0; L^{4})$.
\end{lem}

\vspace{5pt}

\begin{proof}
We multiply $2w$ by (3.1) and observe that

\begin{align*}
\frac{\dd}{\dd t}\int_{\Omega}|w|^{2}\dd x+2\int_{\Omega}|\nabla w|^{2}\dd x=-2\int_{\Omega}(w\cdot \nabla v)\cdot w\dd x-2\int_{\Omega}(v\cdot \nabla v)\cdot w\dd x.
\end{align*}\\
Applying the Young's inequality implies that

\begin{align*}
\Bigg| 2\int_{\Omega}(w\cdot \nabla v)\cdot w\dd x\Bigg|
&=\Bigg| 2\int_{\Omega}(w\cdot \nabla w)\cdot v\dd x\Bigg| \\
&\leq 2||w||_{2}||v||_{\infty}||\nabla w||_{2}\\
&\leq ||w||_{2}^{2}||v||_{\infty}^{2}+||\nabla w||_{2}^{2},\\
\Bigg| 2\int_{\Omega}(v\cdot \nabla v)\cdot w\dd x\Bigg|
&\leq ||v||_{\infty}^{2}||w||_{2}^{2}+||\nabla v||_{2}^{2}.
\end{align*} \\
It follows that

\begin{align*}
\frac{\dd}{\dd t}\int_{\Omega}|w|^{2}\dd x+\int_{\Omega}|\nabla w|^{2}\dd x
\leq 2N^{2}\int_{\Omega}|w|^{2}\dd x+N^{2}.
\end{align*}\\
We set

\begin{align*}
\varphi(t)=\int_{\Omega}|w|^{2}\dd x,
\end{align*}\\
and observe that $\varphi$ satisfies $\varphi(0)=0$ and  

\begin{align*}
\frac{\dd}{\dd t}(\varphi e^{-2N^{2}t})
&=(\dot{\varphi}-2N^{2}\varphi)e^{-2N^{2}t}\\
&\leq N^{2}e^{-2N^{2}t}.
\end{align*}\\
Integrating the both sides between $(0,t)$ yields that  

\begin{align*}
\varphi(t)\leq \frac{1}{2}(e^{2N^{2}t}-1).
\end{align*}\\
It follows that 

\begin{align*}
\int_{\Omega}|w|^{2}\dd x+\int_{0}^{t}\int_{\Omega}|\nabla w|^{2}\dd x\dd s
&\leq
2N^{2}\int_{0}^{t}\varphi(s)\dd s+N^{2}t \\
&\leq N^{2}\int_{0}^{t}(e^{2N^{2}s}-1) \dd s+N^{2}t  \\
&=\frac{1}{2}(e^{2N^{2}t}-1).
\end{align*}\\
We proved (4.4). By the Sobolev inequality \cite[Chapter1, Lemma 1]{Lady61}, we estimate

\begin{align*}
||w||_{4}\leq C_s||w||_{2}^{\frac{1}{2}}||\nabla w||_{2}^{\frac{1}{2}}.
\end{align*}\\
By the global energy estimate (4.4), $w\in L^{4}(0,T_0; L^{4})$ follows. 
\end{proof}

\vspace{10pt}
\subsection{$L^{\infty}$-bounds}

The global bound in $L^{4}(0,T_0; L^{4})$ implies:

\vspace{10pt}

\begin{prop}
The mild solution $w$ of (3.1) for $w_0=0$ is globally bounded in $\overline{\Omega}\times [0,T_0]$.
\end{prop}

\vspace{5pt}

\begin{proof}
We observe that the mild solution $w$ is locally bounded in $L^{4}$ by (4.1). We shall show the global bound  

\begin{align*}
w\in C([0,T_0]; L^{4}).    \tag{4.5}
\end{align*}\\
The global bound (4.5) implies that $w\in C([0,T_0]; L^{\infty})$. In fact, by the local solvability of the perturbed system on $L^{4}$ (Lemma A.1), for each $t_0\in (0,T_0)$, we have 

\begin{align*}
(t-t_0)^{\frac{1}{2}}\nabla w\in C([t_0,t_0+T_1]; L^{4}),  \tag{4.6}
\end{align*} \\
for $T_{1}\geq \varepsilon_4 K_{4}^{-4}$ with the constants

\begin{align*}
K_{4}&=||w||_{L^{4}}(t_0)+N^{\frac{1}{2}}(t_0),\\
N(t_0)&=\sup_{t_0< t \leq T_0}\Big\{||v||_{L^{\infty}\cap \h}(t)+(t-t_0)^{\frac{1}{2}}||\nabla v||_{L^{\infty}}(t)  \Big\}.
\end{align*}\\
The constant $N(t_0)$ is bounded for $t_0\in (0,T_0]$ by (3.2). The global bound (4.5) implies that $K_4$ is bounded for all $t_0\in (0,T_0]$. Thus the constant $T_1$ is uniformly estimated from below for $t_0\in (0,T_0]$. Since $\nabla w$ is bounded in $L^{4}$ near $t=0$ by (4.1), it follows from (4.6) that 
 
\begin{align*}
\nabla w\in C([0,T_0]; L^{4}).
\end{align*}\\
By the Sobolev inequality \cite[Lemma 3.1.4]{Lunardi}, we estimate 

\begin{align*}
||w||_{\infty}\leq C_{s}||w||_{4}^{\frac{1}{2}}||\nabla w||_{4}^{\frac{1}{2}}.
\end{align*}\\
Thus, $w\in C([0,T_0]; L^{\infty})$ follows.\\

We prove (4.5). Suppose that there exists some $T_*\in (0,T_0]$ such that $w$ blows up at $t=T_{*}$ on $L^{4}$. Then, applying Corollary A.2 implies that 

\begin{align*}
||w||_{4}(t)+N^{\frac{1}{2}}(t)\geq \frac{\varepsilon_4'}{(T_*-t)^{\frac{1}{4}}}\quad \textrm{for}\ t<T_*,
\end{align*}\\
with some constant $\varepsilon_4'$. Since $N(t)$ is bounded for all $t\in [0,T_0]$ by (3.2), we have

\begin{align*}
\int_{0}^{T_{*}}||w||_{4}^{4}\dd t=\infty.
\end{align*}\\
This contradicts $w\in L^{4}(0,T_0; L^{4})$ by Lemma 4.3. We reached a contradiction. The proof is complete. 
\end{proof}

\vspace{5pt}

It remains to show that $\nabla w$ is globally bounded in $L^{2}$.

\vspace{5pt}

\begin{thm}
The mild solution of (3.1) for $w_0=0$ is globally bounded in $L^{\infty}\cap H^{1}$, i.e., $w\in C([0,T_0]; L^{\infty}\cap H^{1})$.
\end{thm}

\vspace{5pt}

\begin{proof}
We set

\begin{align*}
J(T)&=\sup_{0<t\leq T}||\nabla w||_{2}(t).
\end{align*}\\
The constant $J=J(T)$ is finite up to some $T\in (0,T_0]$ by Theorem 3.1. We show that 

\begin{align*}
J(T_0)<\infty.
\end{align*}\\
Since $w$ is globally bounded on $L^{2}\cap L^{\infty}$ by Lemma 4.3 and Proposition 4.4, we have 

\begin{align*}
R(T_0)=\sup_{0<t\leq T_0}\{||w||_{2}+||w||_{\infty}\}<\infty.
\end{align*}\\
We separate $\nabla w$ into two terms and estimate 

\begin{align*}
||\nabla w||_{2}
&\leq \int_{0}^{t}||\nabla S(t-s)\p \tilde{w}\cdot \nabla \tilde{w}||_{2}\dd s \\
&\leq \int_{0}^{t}||\nabla S(t-s)\p \tilde{w}\cdot \nabla w||_{2}\dd s+\int_{0}^{t}||\nabla S(t-s)\p \tilde{w}\cdot \nabla v||_{2}\dd s\\
&=:I+II.
\end{align*}\\
We estimate $I$. By (3.4) and a duality, we have 

\begin{align*}
||S(t)\p\D\ F||_{2}\leq \frac{C_2}{t^{\frac{1}{2}}}||F||_{2} \quad \textrm{for}\ F\in H^{1},\ t\leq T_0. \tag{4.7}
\end{align*}\\
It follows from (3.4) and (4.7) that 

\begin{align*}
||\nabla S(t-s)\p\tilde{w}\cdot \nabla w||_{2}
&=\Bigg\|\nabla S\Bigg(\frac{t-s}{2}\Bigg)S\Bigg(\frac{t-s}{2}\Bigg)\p\D\ (\tilde{w}w)\Bigg\|_{2} \\
&\leq \frac{\sqrt{2}C_2}{(t-s)^{\frac{1}{2}}}\Bigg\|S\Bigg(\frac{t-s}{2}\Bigg)\p\D\ (\tilde{w}w)\Bigg\|_{2}\\
&\leq \frac{2C_2^{2}}{t-s}||\tilde{w}||_{\infty}||w||_{2}. 
\end{align*}\\
Since the right-hand side is not integrable near $s=t$, we use the constant $J=J(T)$ and estimate $I$. Applying (3.4) implies that

\begin{align*}
||\nabla S(t-s)\p\tilde{w}\cdot \nabla w||_{2}
\leq \frac{C_2}{(t-s)^{\frac{1}{2}}}||\tilde{w}||_{\infty}||\nabla w||_{2}.
\end{align*}\\
By combining the above two estimates, we have   

\begin{align*}
||\nabla S(t-s)\p\tilde{w}\cdot \nabla w||_{2}
&=||\nabla S(t-s)\p\tilde{w}\cdot \nabla w||_{2}^{\frac{1}{2}}
||\nabla S(t-s)\p\tilde{w}\cdot \nabla w||_{2}^{\frac{1}{2}} \\
&\leq \Bigg(\frac{2C_2^{2}}{t-s}||\tilde{w}||_{\infty}||w||_{2}   \Bigg)^{\frac{1}{2}}\Bigg(\frac{C_2}{(t-s)^{\frac{1}{2}}}||\tilde{w}||_{\infty}||\nabla w||_{2}   \Bigg)^{\frac{1}{2}}\\
&\leq \frac{\sqrt{2}C_{2}^{\frac{3}{2}}}{(t-s)^{\frac{3}{4}}}||\tilde{w}||_{\infty}||w||_{2}^{\frac{1}{2}}||\nabla w||_{2}^{\frac{1}{2}}.
\end{align*}\\
We estimate 

\begin{align*}
I
&\leq \int_{0}^{t}\frac{\sqrt{2}C_{2}^{\frac{3}{2}}}{(t-s)^{\frac{3}{4}}}||\tilde{w}||_{\infty}||w||_{2}^{\frac{1}{2}}||\nabla w||_{2}^{\frac{1}{2}}\dd s\\
&\leq \sqrt{2}C_{2}^{\frac{3}{2}}(R+N)R^{\frac{1}{2}}J^{\frac{1}{2}}\int_{0}^{t}\frac{\dd s}{(t-s)^{\frac{3}{4}}}\\
&=4\sqrt{2}C_2^{\frac{3}{2}}(R+N)R^{\frac{1}{2}}J^{\frac{1}{2}}T^{\frac{1}{4}},
\end{align*}\\
with the constant $N$ in Theorem 3.1. Here, the time variables are  suppressed for $J=J(T)$ and $R=R(T_0)$. We estimate $II$ by (3.4). It follows that

\begin{align*}
II
&\leq \int_{0}^{t}\frac{C_2}{(t-s)^{\frac{1}{2}}}||\tilde{w}\cdot \nabla v||_{2}\dd s \\
&\leq 2C_2 (R+N)NT^{\frac{1}{2}}.
\end{align*}\\
By combining the estimates for $I$ and $II$, we obtain 

\begin{align*}
J\leq C(R+N)(R^{\frac{1}{2}}J^{\frac{1}{2}}T^{\frac{1}{4}}_{0}+NT^{\frac{1}{2}}_{0}   ),
\end{align*}\\
with some constant $C$. Applying Young's inequality implies that 

\begin{align*}
J\leq C^{2}(R+N)^{2}RT^{\frac{1}{2}}_{0}+2C(R+N)NT^{\frac{1}{2}}_{0}.
\end{align*}\\
Thus $J=J(T)$ is bounded for all $T\leq  T_{0}$. We proved that $\nabla w\in C([0,T_0]; L^{2})$. The proof is now complete.
\end{proof}

\vspace{10pt}

\section{Mild solutions on $L^{\infty}\cap \dot{H}^{1}$}

\vspace{10pt}

Now, we construct global solutions of (1.1) for $u_0\in L^{\infty}_{\sigma}\cap \h_{0}$. It suffices to show a global bound for local solutions on $L^{\infty}\cap \h$

\vspace{10pt}

\begin{proof}[Proof of Theorem 1.1]
We apply Theorem 3.1 for $v=0$ and observe that for $u_0\in L^{\infty}_{\sigma}\cap \H$, there exists 

\begin{align*}
T\geq \frac{\varepsilon}{||u_0||_{L^{\infty}\cap \dot{H}^{1}}^{2}}
\end{align*}\\
and a unique mild solution $u\in C_{w}([0,T]; L^{\infty})$ of (1.1) satisfying $t^{1/2}\nabla u\in C_{w}([0,T]; L^{\infty})$ and $\nabla u\in C_{w}([0,T]; L^{2})$. We set 

\begin{align*}
u(t)&=S(t)u_{0}-\int_{0}^{t}S(t-s)\p (u\cdot \nabla u)(s) \dd s\\
&=: u_1+u_2.
\end{align*}\\
We observe that $v=u_1$ satisfies the condition (3.2) by Corollary 2.2 and $u_2$ is a mild solution of (3.1) for $w_0=0$, i.e., 
\begin{align*}
u_2(t)&=-\int_{0}^{t}S(t-s)\p (u\cdot \nabla u)(s) \dd s,\quad u=u_2+v.
\end{align*}\\
Since $u_2$ is globally bounded in $L^{\infty}\cap H^{1}$ by Theorem 4.5, the local solution $u$ is globally bounded $L^{\infty}\cap \dot{H}^{1}$ and extendable for all $t>0$. The proof is now complete.
\end{proof}

\vspace{10pt}

\begin{rems}
\noindent
(i) (Regularity)
The mild solutions constructed in Theorem 1.1 are H\"older continuous up to first derivatives, i.e., $\nabla u\in C^{\mu,\mu/2} (\overline{\Omega}\times [\delta,T])$ for $\delta,T>0$ and $\mu\in (0,1)$. Since the regularizing estimates 

\begin{equation*}
\sup_{0< t\leq T_1}\left\{\big\|u\big\|_{L^{\infty}}(t)+t^{\frac{1}{2}}\big\|\nabla u\big\|_{L^{\infty}}(t)+t^{\frac{1+\beta}{2}}\Big[\nabla u\Big]^{(\beta)}_{\Omega}(t)  \right\}\leq C_1\big\|u_0\big\|_{L^{\infty}},       
\end{equation*}
\begin{equation*}
\sup_{x\in \Omega}\left\{\Big[u\Big]^{(\gamma)}_{[\delta,T_1]}(x)+\Big[\nabla u\Big]^{(\frac{\gamma}{2})}_{[\delta,T_1]}(x) \right\}\leq C_2\big\|u_0\big\|_{L^{\infty}},   
\end{equation*}\\
hold for $\beta, \gamma\in (0,1)$, $\delta\in (0,T_1)$ and $T_1\geq \varepsilon/||u_0||_{\infty}^{2}$ with some constants $\varepsilon$, $C_1$ and $C_2$ \cite{A5}, applying the above estimates for $u\in C_{w}([0,\infty); L^{\infty})$ and each $t>0$ implies the H\"older continuity of mild solutions. Here, $[\cdot]_{\Omega}^{(\beta)}$ denotes the $\beta$-th H\"older semi-norm in $\overline{\Omega}$. The mild solution satisfies (1.1) in the sense that 

\begin{equation*}
\int_{0}^{\infty}\int_{\Omega}\big(u\cdot(\partial_{t}\varphi+\Delta \varphi)+u u: \nabla\varphi  \big)\dd x\dd t=-\int_{\Omega}u_0(x)\cdot \varphi (x,0)\dd x     
\end{equation*} \\
for all $\varphi\in C^{\infty}_{c,\sigma}(\Omega\times [0,\infty))$. We observe that the second derivatives are in $L^{2}_{\textrm{loc}}(\overline{\Omega})$ for a.e. $t\in (0,\infty)$. We set $u=u_1+u_2$ as in the proof of Theorem 1.1 and invoke that $u_1\in C^{2+\mu,1+\mu/2}(\overline{\Omega}\times [\delta,T])$ satisfies the Stokes equations in a classical sense \cite{AG2}. Since $u\in C_{w}([0,\infty); L^{\infty})$ satisfies $\nabla u\in C_{w}([0,\infty); L^{2})$, the maximal regularity estimate \cite{deSimon64} (\cite{Sohr})  implies that $\partial_t u_2, \nabla^{2} u_2\in L^{s}(0,T; L^{2})$ for $s\in (1,\infty)$ and each $T>0$. Thus $u_2$ satisfies the inhomogeneous Stokes equations

\begin{align*}
\partial_t u_2-\Delta u_2+\nabla p_2=-u\cdot \nabla u \quad \textrm{on}\ L^{2}
\end{align*}\\
for a.e. $t\in (0,\infty)$ with the associated pressure $p_2$. Thus $\partial_t u$, $\nabla^{2}u\in L^{2}_{\textrm{loc}}(\overline{\Omega})$ for a.e. $t\in (0,\infty)$.

\noindent
(ii) (Associated pressure) We multiply the projection $\mathbb{Q}=I-\mathbb{P}: L^{2}\longrightarrow L^{2}$ by (1.1) and observe that the associated pressure of a mild solution $u$ is expressed by 

\begin{align*}
\nabla p=\mathbb{Q}\Delta u-\mathbb{Q}u\cdot \nabla u.
\end{align*}\\
Since the mild solution $u\in C_{w}([0,\infty); L^{\infty})$ satisfies $\nabla u\in C_{w}([0,\infty); L^{2})$, the second term is defined as an element of $L^{2}$ for each $t>0$. Although the first term may not be defined for mild solutions on $L^{\infty}\cap \h$, we are able to define the associated pressure by using the solution operator $\mathbb{K}: L^{\infty}_{\textrm{tan}}(\partial\Omega)\longrightarrow L^{\infty}_{d}(\Omega)$ of the homogeneous Neumann problem 

\begin{align*}
\Delta q&=0\quad \textrm{in}\ \Omega\\
\frac{\partial q}{\partial n}&=\textrm{div}_{\partial\Omega}\ W\quad \textrm{on}\ \partial\Omega.
\end{align*}\\
Here, $L^{\infty}_{\textrm{tan}}(\partial\Omega)$ denotes the space of all bounded tangential vector fields on $\partial\Omega$ and $L^{\infty}_{d}(\Omega)$ denotes the space of all locally integrable functions $f$ in $\Omega$ such that $df\in L^{\infty}(\Omega)$ for $d(x)=\inf_{y\in \partial\Omega}|x-y|$ and $x\in \Omega$. The symbol $\textrm{div}_{\partial\Omega}$ denotes the surface divergence on $\partial\Omega$. Note that $\Delta u\cdot n=\textrm{div}_{\partial\Omega}\ W$ for $W=\omega n^{\perp}$ and $\omega=\partial_1 u^{2}-\partial_2 u^{1}$, where $n=(n^{1},n^{2})$ denotes the unit outward normal vector field on $\partial\Omega$ and $n^{\perp}=(n^{2},-n^{1})$. The associated pressure of (1.1) is then expressed by 

\begin{align*}
\nabla p=\mathbb{K}W-\mathbb{Q}u\cdot \nabla u
\end{align*}\\
for the mild solution $u$ on $L^{\infty}\cap \h$. See also \cite[Remarks 1.2 (v)]{A5}.
\end{rems}

\vspace{10pt}

\section{Asymptotic behavior at the space infinity}

\vspace{10pt}

We prove Theorem 1.3. We show that the Stokes flow converges to a constant as $|x|\to\infty$ for asymptotically constant initial data $u_0\in BUC_{\sigma}\cap \h_{0}$. We subtract initial data from the Stokes flow and estimate a spatial decay of the difference by using a fractional power of the Stokes operator on $L^{2}$. After the proof of Theorem 1.3, we remark on a half space case (Remarks 6.4). 

\vspace{10pt}

\begin{lem}
Let $u_0\in BUC_{\sigma}\cap \h_{0}$ satisfy

\begin{align*}
\lim_{R\to\infty}\sup_{|x|\geq R}|u(x)-u_{\infty}|=0
\end{align*}\\
for some constant $u_{\infty}\in \mathbb{R}^{2}$. Then, 

\begin{align*}
\lim_{R\to\infty}\sup_{|x|\geq R}|S(t)u_0-u_{\infty}|=0\quad \textrm{for each}\ t\geq 0.
\end{align*}\\
\end{lem}

\vspace{5pt}

We observe that $S(t)u_0-u_0$ decays as $|x|\to\infty$.

\vspace{5pt}

\begin{prop}
The estimate
 
\begin{align*}
||S(t)u_0-u_0||_{L^{2}}\leq Ct^{\frac{1}{2}}||u_0||_{L^{\infty}\cap \h}  \tag{6.1}
\end{align*}\\
holds for $u_0\in L^{\infty}_{\sigma}\cap \h_{0}$ and $t>0$ with some constant $C$. If in addition that $u_0\in BUC_{\sigma}$, then $S(t)u_0-u_0\in BUC\cap L^{2}$ for each $t\geq 0$.
\end{prop}

\vspace{5pt}

\begin{proof}
We first show (6.1) for $u_0\in L^{\infty}_{\sigma}\cap H^{1}_{0} $ with compact support in $\overline{\Omega}$. By the analyticity of $S(t)$ and an estimate of a fractional power on $L^{2}$ \cite[III. 2.2.1 Lemma]{Sohr}, we have 

\begin{align*}
||S(t)u_0-u_0||_{L^{2}}
&\leq \int_{0}^{t}||AS(s)u_0||_{L^{2}}\dd s\\
&\leq \int_{0}^{t}||A^{\frac{1}{2}}S(s)A^{\frac{1}{2}}u_0||_{L^{2}}\dd s\\
&\leq \int_{0}^{t}\frac{1}{s^{\frac{1}{2}}}||A^{\frac{1}{2}}u_0||_{L^{2}}\dd s= 2t^{\frac{1}{2}}||\nabla u_0||_{L^{2}}\quad \textrm{for}\ t>0.
\end{align*}\\
Thus (6.1) holds. For general $u_0\in L^{\infty}_{\sigma}\cap \h_0$, we take a sequence $\{u_{0,m}\}\subset L^{\infty}_{\sigma}\cap\h_0$ with compact support in $\overline{\Omega}$ satisfying (2.4) by Lemma 2.4 and obtain (6.1) by approximation.
\end{proof}

\vspace{5pt}

Proposition 6.2 implies the pointwise convergence $S(t)u_0-u_0\to 0$ as $|x|\to\infty$. We prepare the following Proposition 6.3 and then give a proof for Lemma 6.1.

\vspace{5pt}

\begin{prop}
Assume that $f\in BUC\cap L^{2}(\Omega)$. Then, 

\begin{align*}
\lim_{R\to\infty}\sup_{|x|\geq R}|f(x)|=0.
\end{align*}\\
\end{prop}

\begin{proof}
Suppose on the contrary that 

\begin{align*}
\overline{\lim_{R\to\infty}}\sup_{|x|\geq R}|f(x)|=\delta>0.
\end{align*}\\
Then, there exists a sequence $\{R_m\}$ such that 

\begin{align*}
\sup_{|x|\geq R_{m}}|f(x)|\geq \frac{\delta}{2},
\end{align*}\\
and $R_m\to\infty$ as $m\to\infty$. We then take $x_m\in \Omega$ such that $|x_m|\geq R_m$ and 

\begin{align*}
|f(x_m)|
&\geq \frac{1}{2}\sup_{|x|\geq R_{m}}|f(x)| \\
&\geq \frac{\delta}{4}.
\end{align*}\\
Since $f$ is uniformly continuous in $\overline{\Omega}$, for arbitrary $\varepsilon>0$ there exists $\eta>0$ such that 

\begin{align*}
|f(x)-f(y)|\leq \varepsilon\quad \textrm{for}\  x, y\in \Omega\quad  \textrm{satisfying}\ |x-y|\leq \eta.
\end{align*}\\
We take $\varepsilon=\delta/8$ and estimate 

\begin{align*}
|f(y)|
&\geq \big| |f(y)-f(x_m)|-|f(x_m)|\big|\\
&\geq \frac{\delta}{8}\quad \textrm{for}\ y\in B_{x_m}(\eta).
\end{align*}\\
The constant $\eta$ depends on $\delta$ and is independent of $m\geq 1$. By choosing a subsequence if necessary, we may assume that $B_{x_m}(\eta)\cap  B_{x_l}(\eta)\neq \emptyset$ for $m\neq l$. It follows that
 
\begin{align*}
\infty>\int_{\Omega}|f|^{2}\dd x
&\geq \sum_{m=1}^{N}\int_{B_{x_m}(\eta)}|f|^{2}\dd x\\
&\geq \Bigg(\frac{\delta}{8}\Bigg)^{2}|B_{0}(\eta)|N\to\infty\quad \textrm{as}\ N\to\infty.
\end{align*}\\
We reached a contradiction. The proof is complete. 
\end{proof}

\vspace{5pt}

\begin{proof}[Proof of Lemma 6.1]
For $u_0\in BUC_{\sigma}\cap \h_0$ satisfying $u_{0}\to u_{\infty}$ as $|x|\to\infty$, we observe that $S(t)u_0-u_0\in BUC\cap L^{2}$ by Proposition 6.2. Applying Proposition 6.3 implies that 

\begin{align*}
\lim_{R\to\infty}\sup_{|x|\geq R}|S(t)u_0-u_{0}|=0\quad \textrm{for each}\ t\geq 0.
\end{align*}\\
It follows that 
\begin{align*}
\sup_{|x|\geq R}|S(t)u_0-u_{\infty}|
&\leq \sup_{|x|\geq R}|S(t)u_0-u_{0}|
+\sup_{|x|\geq R}|u_0-u_{\infty}|\\
&\to 0\quad \textrm{as}\ R\to\infty.
\end{align*}
\end{proof}

\vspace{5pt}

We now complete: 

\vspace{5pt}

\begin{proof}[Proof of Theorem 1.3]
We set 

\begin{align*}
u(t)&=S(t)u_{0}-\int_{0}^{t}S(t-s)\p (u\cdot \nabla u)(s) \dd s\\
&=: u_1+u_2\quad \textrm{for}\ t\geq 0.
\end{align*}\\
It suffices to show that 

\begin{align*}
u_2(\cdot,t)\in BUC\cap L^{2}(\Omega)\quad \textrm{for each}\ t\geq 0.   \tag{6.2}
\end{align*}\\
In fact, the condition (6.2) implies that 

\begin{align*}
\lim_{R\to\infty}\sup_{|x|\geq R}|u_{2}(x,t)|=0,
\end{align*}\\
by Proposition 6.3. It follows from Lemma 6.1 that

\begin{align*}
\sup_{|x|\geq R}|u(x,t)-u_{\infty}|
&\leq \sup_{|x|\geq R}|u_1(x,t)-u_{\infty}|
+\sup_{|x|\geq R}|u_2(x,t)|\\
&\to 0\quad \textrm{as}\ R\to\infty.
\end{align*}\\
Thus the assertion of Theorem 1.3 holds. It remains to show (6.2). We set $f=-u\cdot \nabla u$ and observe that 

\begin{align*}
f\in L^{\infty}(0,T_0; L^{2})\quad \textrm{for}\ T_0>0,
\end{align*}\\
by $u\in C([0,T_0]; BUC)$ and $\nabla u\in C_{w}([0,T_0]; L^{2})$. By the same way as the proof of Lemma 4.2, applying (4.3) implies that 

\begin{align*}
||u_2||_{L^{p}}\leq Ct^{\frac{1}{2}+\frac{1}{p}}||f||_{L^{\infty}(0,T_0; L^{2})}
\end{align*}\\
for $2\leq p\leq \infty$. Moreover, we have $\nabla u_2\in C([0,T_0]; L^{q})$ for $2\leq q<\infty$. Thus $u_2\in BUC\cap L^{2}$ for each $t\geq 0$. We proved (6.2). The proof is now complete.
\end{proof}

\vspace{10pt}

\begin{rems}

\noindent
(A half space case) (i) The statement of Theorem 1.1 is valid also for a half space. Since the Stokes semigroup is a bounded analytic semigroup on $L^{\infty}_{\sigma}\cap \h_{0}$ as in Remarks 2.6 (ii), the proof of Theorem 1.1 is valid also for a half space without modifications. 

\noindent
(ii) The space $\dot{H}^{1}_{0}(\mathbb{R}^{2}_{+})$ consists of decaying functions as $|x|\to\infty$ in the sense that 

\begin{align*}
\lim_{r\to\infty}\int_{0}^{\pi}|u(re_{r})|^{2}\dd \theta=0.  \tag{6.3}
\end{align*}\\
See \cite[Proposition 5.6]{GWittwer}. For the reader's convenience, we outline the proof given in \cite{GWittwer}. We first observe that $u/x_2$ belongs to $L^{2}(\mathbb{R}^{2}_{+})$ by the Hardy's inequality \cite[2.7.1]{Mazya}

\begin{align*}
\Big\|\frac{u}{x_2}\Big\|_{L^{2}(\mathbb{R}^{2}_{+})}\leq 2||\nabla u||_{L^{2}(\mathbb{R}^{2}_{+})}
\end{align*}\\
for $u\in \dot{H}^{1}_{0}(\mathbb{R}^{2}_{+})$. We set $D_{r}=(B_{0}(2r)\backslash \overline{B_{0}(r)})\cap \mathbb{R}^{2}_{+}$ for $r>0$. Since the $L^{2}$-trace on $\partial B_{0}(1)\cap \mathbb{R}^{2}_{+}$ is estimated by the $H^{1}$-norm in $D_1$, by dilation we have 

\begin{align*}
\frac{1}{r^{\frac{1}{2}}}||u||_{L^{2}(\partial B_{0}(r)\cap \mathbb{R}^{2}_{+}  )}\leq C\Big(\frac{1}{r}||u||_{L^{2}( D_r)}+||\nabla u||_{L^{2}( D_r)}\Big).
\end{align*} \\
Since $x\in D_r$ satisfies $r<|x|<2r$ and 

\begin{align*}
\frac{1}{r}\leq \frac{2}{|x|}\leq \frac{2}{x_2},
\end{align*}\\
we estimate

\begin{align*}
\frac{1}{r}||u||_{L^{2}(D_r)}\leq 2\Big\|\frac{u}{x_2}\Big\|_{L^{2}(D_r)}.
\end{align*}\\
Hence we have 

\begin{align*}
\lim_{r\to\infty}\frac{1}{r^{\frac{1}{2}}}||u||_{L^{2}(\partial B_{0}(r)\cap \mathbb{R}^{2}_{+})}=0.
\end{align*}\\
Transforming this by the polar coordinate implies (6.3). 
\end{rems}

\vspace{15pt}
\appendix

\section{Local solvability of the perturbed system on $L^{p}$}

\vspace{10pt}

In Appendix A, we prove local solvability of the perturbed system (3.1) on $L^{p}$ ($p>2$) and the blow-up estimate used in Section 4. The proof is by an iterative argument and parallel to that of Theorem 3.1.

\vspace{5pt}

\begin{lem}
Let $v$ be a solenoidal vector field in $\Omega$ satisfying (3.2). For $p\in (2,\infty)$, there exists a constant $\varepsilon_{p}$ such that for $w_0\in L^{p}_{\sigma}$, there exists $T\geq \varepsilon_{p} K^{-2p/(p-2)}_{p}$ for 

\begin{equation*}
\begin{aligned}
K_{p}&=||w_0||_{L^{p}}+N^{1-\frac{2}{p}},\\
N&=\sup_{0< t\leq T_0}\Big\{||v||_{L^{\infty}\cap \h}(t)+t^{1/2}||\nabla v||_{L^{\infty}}(t)  \Big\},
\end{aligned}
\tag{A.1}
\end{equation*}\\
and a unique mild solution $w\in C([0,T]; L^{p})$ of (3.1) satisfying $t^{1/2}\nabla w\in C([0,T]; L^{p})$.
\end{lem}

\vspace{5pt}

\begin{cor}
Let $w\in C([0,T_*); L^{p})$ be a mild solution of (3.1) for $w_0\in L^{p}_{\sigma}$. Suppose that $w$ blows up at $t=T_*$ on $L^{p}$. Then, 

\begin{align*}
||w||_{L^{p}}(t)+N(t)^{1-\frac{2}{p}}\geq \frac{\varepsilon_p'}{(T_*-t)^{\frac{1}{2}-\frac{1}{p} }} \quad \textrm{for}\ t<T_*,   \tag{A.2}
\end{align*}\\
where  

\begin{align*}
N(t)=\sup_{t< s\leq T_0}\Big\{||v||_{L^{\infty}\cap \h}(s)+(s-t)^{1/2}||\nabla v||_{L^{\infty}}(s)  \Big\},  
\end{align*}\\
and $\varepsilon_{p}^{'}=\varepsilon_{p}^{1/2-1/p}$. 
\end{cor}

\begin{proof}
We fix an arbitrary $t_0\in (0,T_*)$. By Lemma A.1, there exists $T_1\geq \varepsilon_p K_{p}^{-2p/(p-2)}$ for $K_{p}=||w||_{p}(t_0)+N(t_0)^{1-2/p}$ and a unique mild solution $w$ in $[t_0,t_0+T_1]$. Since $w$ blows up at $t=T_*$, the existence time $T_1$ is smaller than $T_{*}-t_0$ and (A.2) holds.
\end{proof}

\vspace{10pt}

In order to prove Lemma A.1, we prepare $L^{p}$-estimates of the Stokes semigroup.

\vspace{10pt}

\begin{prop}
Let $p\in [2,\infty)$ and $T_0>0$. There exist constants $C_1$ and $C_2$ such that the estimates

\begin{align*}
t^{\frac{|k|}{2}}||\partial_x^{k}S(t)\p w_0||_{p}&\leq C_1||w_0||_{p},  \tag{A.3} \\
t^{\frac{1}{2}-\frac{1}{p}+\frac{|k|}{2}}||\partial_x^{k}S(t)\p f||_{p}&\leq C_2||f||_{2},     \tag{A.4}
\end{align*}\\
hold for $w_0\in L^{p}$, $f\in L^{2}$, $|k|\leq 1$ and $0< t\leq T_0$.
\end{prop}

\vspace{5pt}

\begin{proof}
The estimate (A.3) follows from the analyticity of the Stokes semigroup \cite{Sl76}, \cite{G81} and the boundedness of the Helmholtz projection on $L^{p}$ \cite{SiSo}. It follows from (3.6) and (A.3) that 

\begin{align*}
||S(t)\p f||_{p}
&\leq ||S(t)\p f||_{\infty}^{\sigma}||S(t)\p f||_{2}^{1-\sigma} \\
&\leq \frac{C}{t^{\frac{\sigma}{2}}}||f||_{2},
\end{align*}\\
for $\sigma=1-2/p$. By combining (A.3), we obtain (A.4) for $|k|\leq 1$.
\end{proof}

\vspace{5pt}

\begin{proof}[Proof of Lemma A.1]
We set the sequence $\{w_{j}\}$ by 

\begin{align*}
w_{j+1}&=S(t)w_0-\int_{0}^{t}S(t-s)\p (w_{j}\cdot \nabla w_{j}+v\cdot \nabla w_{j}+w_j\cdot \nabla v+v\cdot \nabla v)\dd s,\\
w_{1}&=S(t)w_0,
\end{align*}\\
and the constants

\begin{align*}
K_{j}=\sup_{0\leq t\leq T_0}\Big\{||w_j||_{p}(t)+t^{\frac{1}{2}}||\nabla w_j||_{p}(t)   \Big\},\quad \textrm{for}\ p>2.
\end{align*}\\
By the Sobolev inequality \cite[Lemma 3.1.4]{Lunardi}, we estimate 

\begin{align*}
||w_j||_{\infty}
\leq \frac{C_3}{t^{\frac{1}{p}}}K_j\quad \textrm{for}\ 0<t\leq T.   \tag{A.5}
\end{align*}\\
We may assume that $T\leq T_0$. We first show that
 
\begin{align*}
K_j\leq 2C_1 K_0  \quad \textrm{for}\ j=1,2,\cdots,  \tag{A.6}
\end{align*}\\
and 

\begin{equation*}
\begin{aligned}
&T\leq \varepsilon_0K_{0}^{-\frac{2p}{p-2}},\\
&K_0=||w_0||_{p}+N^{1-\frac{2}{p}},
\end{aligned}
\tag{A.7}
\end{equation*}\\
with some constant $\varepsilon_0$. We apply (A.3) and (A.5) to estimate 

\begin{align*}
||S(t-s)\p (w_j\cdot \nabla w_j)||_{p}
&\leq C_1||w_j\cdot \nabla w_{j}||_{p} \\
&\leq \frac{C_1C_3}{s^{\frac{1}{p}+\frac{1}{2}}}K_j^{2}, \\
||S(t-s)\p (v\cdot \nabla w_j)||_{p}
&\leq C_1||v\cdot \nabla w_{j}||_{p} \\
&\leq \frac{C_1}{s^{\frac{1}{2}}}N K_j.
\end{align*}\\
It follows from (A.4) and (A.5) that 

\begin{align*}
||S(t-s)\p (w_j\cdot \nabla v)||_{p}
&\leq \frac{C_2}{(t-s)^{\frac{1}{2}-\frac{1}{p}}}||w_j\cdot \nabla v||_{2} \\
&\leq \frac{C_2C_3}{(t-s)^{\frac{1}{2}-\frac{1}{p}} s^{\frac{1}{p}}}K_j N, \\
||S(t-s)\p (v\cdot \nabla v)||_{p}
&\leq \frac{C_2}{(t-s)^{\frac{1}{2}-\frac{1}{p}  }}N^{2}. 
\end{align*}\\
By integrating the above estimates in $(0,t)$, we obtain

\begin{align*}
||w_{j+1}||_{p}\leq ||S(t)w_0||_{p}+C(K_j^{2}T^{\frac{1}{2}-\frac{1}{p}}+NK_{j}T^{\frac{1}{2}}+N^{2}T^{\frac{1}{2}+\frac{1}{p}}).
\end{align*}\\
By a similar way, we estimate 

\begin{align*}
t^{\frac{1}{2}}||\nabla w_{j+1}||_{p}\leq t^{\frac{1}{2}} ||\nabla S(t)w_0||_{p}+C'(K_j^{2}T^{\frac{1}{2}-\frac{1}{p}}+NK_{j}T^{\frac{1}{2}}+N^{2}T^{\frac{1}{2}+\frac{1}{p}}   ).
\end{align*}\\
It follows from (A.3) and (A.7) that 

\begin{align*}
K_{j+1}\leq C_1K_0+C_4(K_j^{2}T^{\frac{1}{2}-\frac{1}{p}}+K_0^{\frac{p}{p-2}}K_{j}T^{\frac{1}{2}}+K_0^{\frac{2p}{p-2}}T^{\frac{1}{2}+\frac{1}{p}}),
\end{align*} \\
with some constant $C_4$. We take a constant $\varepsilon_0>0$ and $T>0$ satisfying (A.7). It follows that

\begin{align*}
K_{j+1}
&\leq C_1K_0+C_4\Big(K_j^{2}K_0^{-1}(K_0T^{\frac{1}{2}-\frac{1}{p}})+K_j(K_0^{\frac{p}{p-2}}T^{\frac{1}{2}})+K_0(K_0^{\frac{p+2}{p-2}}T^{\frac{1}{2}+\frac{1}{p} })\Big) \\
&\leq C_1K_0+C_4(K_j^{2}K_0^{-1}\varepsilon_{0}^{\frac{1}{2}-\frac{1}{p}}+K_j\varepsilon_0^{\frac{1}{2}}+K_0\varepsilon_0^{\frac{1}{2}+\frac{1}{p}}  ).
\end{align*}\\
By a fundamental calculation, we obtain (A.6) with some constant $\varepsilon_0$ satisfying 

\begin{align*}
C_4(4C_1^{2}{\varepsilon_0}^{\frac{1}{2}-\frac{1}{p}}+2C_1{\varepsilon_0}^{\frac{1}{2}}+{\varepsilon_0}^{\frac{1}{2}+\frac{1}{p}} )\leq C_1.
\end{align*}\\
Since $S(t)w_0$ is continuous on $L^{p}$ and $t^{1/2}\nabla S(t)w_0$ vanishes at time zero, $w_j \in C([0,T]; L^{p}_{\sigma})$ satisfies $t^{1/2}\nabla w_j\in C([0,T]; L^{p})$ and $t^{1/2}||\nabla w_{j}||_{p}\to 0$ as $t\to 0$.\\

We show that the sequence $\{w_j\}$ converges to a limit $w\in C([0,T]; L^{p}_{\sigma})$. We set 

\begin{align*}
&\rho_j=w_j-w_{j-1},\\
&L_{j}=\sup_{0\leq t\leq T}\Big\{||\rho_j||_{p}(t)+t^{\frac{1}{2}}||\nabla \rho_j||_{p}(t)   \Big\}.
\end{align*}\\
By the Sobolev inequality, we have 

\begin{align*}
||\rho_j||_{\infty}\leq \frac{C_{3}}{t^{\frac{1}{p}}}L_j\quad \textrm{for}\ 0<t\leq T.  \tag{A.8}
\end{align*}\\
By a similar way as the proof of Theorem 3.1, we estimate 

\begin{align*}
\rho_{j+1}
&=-\int_{0}^{t}S(t-s)\p(\tilde{w}_{j}\cdot \nabla \rho_j+\rho_j\cdot \nabla \tilde{w}_{j-1})\dd s\\
&=-\int_{0}^{t}S(t-s)\p(w_j\cdot \nabla \rho_j+v\cdot \nabla \rho_j+\rho_j\cdot \nabla w_{j-1}+\rho_j\cdot \nabla v)\dd s,
\end{align*}\\
where $\tilde{w}_{j}=w_j+v$. It follows from (A.3)-(A.5) and (A.8) that

\begin{align*}
||S(t-s)\p (w_j\cdot \nabla \rho_j)||_{p}
&\leq C_1 ||w_j\cdot \nabla \rho_j||_{p}\\
&\leq \frac{C_1C_3}{s^{\frac{1}{p}+\frac{1}{2}}}K_jL_j,\\
||S(t-s)\p (v \cdot \nabla \rho_j)||_{p}
&\leq \frac{C_1}{s^{\frac{1}{2}}}NL_j,\\
||S(t-s)\p (\rho_j\cdot \nabla w_{j-1})||_{p}
&\leq \frac{C_1C_3}{s^{\frac{1}{p}+\frac{1}{2}}}L_jK_{j-1},\\
||S(t-s)\p (\rho_j\cdot \nabla v)||_{p}
&\leq \frac{C_2C_3}{(t-s)^{\frac{1}{2}-\frac{1}{p}}s^{\frac{1}{p}} }L_jN.
\end{align*} \\
We integrate the above estimates in $(0,t)$ and obtain

\begin{align*}
||\rho_{j+1}||_{p}\leq C\big((K_j+K_{j-1})L_j T^{\frac{1}{2}-\frac{1}{p}}+NL_j T^{\frac{1}{2}}  \big).
\end{align*} \\
By a similar way, we estimate 

\begin{align*}
t^{\frac{1}{2}}||\nabla \rho_j||_{p}
\leq C' \big((K_j+K_{j-1})L_j T^{\frac{1}{2}-\frac{1}{p}}+NL_j T^{\frac{1}{2}}  \big).
\end{align*}  \\
By (A.6) and (A.7), we obtain 

\begin{align*}
L_{j+1}
\leq C_5 (K_0 T^{\frac{1}{2}-\frac{1}{p}}+K_0^{\frac{p}{p-2}}T^{\frac{1}{2}}  )L_j. 
\end{align*} \\
We set

\begin{align*}
K_{0}^{\frac{2p}{p-2}}T\leq \varepsilon_1
\end{align*}\\
and observe that 

\begin{align*}
L_{j+1}\leq \frac{1}{2}L_{j},
\end{align*}\\
for some constant $\varepsilon_1$ satisfying 

\begin{align*}
C_5(\varepsilon_1^{\frac{1}{2}-\frac{1}{p}}+\varepsilon_1^{\frac{1}{2}} )\leq \frac{1}{2}.
\end{align*}\\
Thus we have

\begin{align*}
L_{j+1}\leq \Big(\frac{1}{2}\Big)^{j}L_{1}\to0\quad \textrm{as}\ j\to\infty,
\end{align*}\\
for
 
\begin{align*}
T\leq \varepsilon_2 K_0^{-\frac{2p}{p-2}}\quad\textrm{and}\quad
\varepsilon_2=\min\{\varepsilon_0, \varepsilon_1\}.
\end{align*}\\
Thus the sequence $\{w_j\}\subset C([0,T]; L^{p}_{\sigma})$ converges to a limit $w\in C([0,T]; L^{p}_{\sigma})$ satisfying $t^{1/2}\nabla w\in C([0,T]; L^{p})$ and $t^{1/2}||\nabla w||_{p}\to 0$ as $t\to0$. The limit $w$ satisfies the integral equation and is unique. The existence time is estimated from below by $T\geq \varepsilon K_{0}^{-2p/(p-2)}$ for 

\begin{align*}
\varepsilon=\frac{\varepsilon_2}{2}.
\end{align*}\\
The proof is now complete.
\end{proof}

\vspace{10pt}

\section{A fractional power of the Stokes operator in a half space}

\vspace{10pt}

In Appendix B, we show $L^{p}$-estimates of a fractional power of the Stokes operator $A$ in a half space.

\vspace{5pt}

\begin{lem}
Let $\Omega=\mathbb{R}^{n}_{+}$, $n\geq 2$. For $p\in (1,\infty)$, there exist constants $C_1$ and $C_2$ such that 

\begin{align*}
||\nabla u||_{p}&\leq C_1||A^{\frac{1}{2}}u||_{p},  \tag{B.1}\\
||A^{\frac{1}{2}}u||_{p}&\leq C_2||\nabla u||_{p},  \tag{B.2}
\end{align*}\\
hold for $u\in D(A^{1/2})$ and $D(A^{1/2})=W^{1,p}_{0,\sigma}$, where $W^{1,p}_{0,\sigma}$ denotes the $W^{1,p}$-closure of $C_{c,\sigma}^{\infty}$.
\end{lem}

\vspace{5pt}

\begin{proof}
The estimate (B.1) is proved in \cite[Theorem 3.6 (ii)]{BM88}. We prove (B.2). It suffices to show (B.2) for $u\in C_{c,\sigma}^{\infty}$ since $W^{1,p}_{0,\sigma}$ is the closure of $C_{c,\sigma}^{\infty}$ in $W^{1,p}$. We use $L^{2}$-theory. Since the Stokes operator is a positive self-adjoint operator on $L^{2}$, by a spectral representation we are able to define the fractional power $A^{1/2}v$ for $v\in D(A^{1/2})=H^{1}_{0,\sigma}$ and we have
 
\begin{align*}
(A^{\frac{1}{2}}u, A^{\frac{1}{2}}v)=(\nabla u,\nabla v).
\end{align*}\\
Here, $(\cdot,\cdot)$ denotes the inner product on $L^{2}$. Moreover, for $\varphi\in C^{\infty}_{c}$ there exists $v\in D(A^{1/2})$ such that $A^{1/2}v=\mathbb{P}\varphi$ \cite[III.2.2.1 Lemma ]{Sohr}. Since the Helmholtz projection $\mathbb{P}$ acts as a bounded operator on $L^{p'}$ \cite{McCracken}, \cite[Theorem 3.1]{BM88}, it follows from (B.1) that 

\begin{align*}
||\nabla v||_{p'}
&\leq C||A^{\frac{1}{2}}v||_{p'}\\
&\leq C'||\varphi||_{p'},
\end{align*}\\
where $p'$ is the H\"older conjugate of $p\in (1,\infty)$. The above estimate implies that

\begin{align*}
|(A^{\frac{1}{2}}u,\varphi )|
&=|(A^{\frac{1}{2}}u, \mathbb{P}\varphi)|\\
&=|(A^{\frac{1}{2}}u,A^{\frac{1}{2}}v) |\\
&=|(\nabla u,\nabla v)|
\leq C'||\nabla u||_{p}||\varphi||_{p'}.
\end{align*}\\
Since $\varphi$ is arbitrary, the estimate (B.2) holds. 
\end{proof}

\vspace{5pt}

\begin{rem}
The estimates (B.1) and (B.2) hold also for bounded domains of class $C^{3}$ in $\mathbb{R}^{n}$, $n\geq 2$. The estimate (B.1) for bounded domains with smooth boundaries is proved in \cite{G85} by estimates of pure imaginary powers of the Stokes operator. Later, more strongly $\mathcal{H}^{\infty}$-calculus is proved in \cite{NS} with $C^{3}$-regularity (see also \cite{GeK}). By the same way as we have seen above, the estimate (B.2) is deduced from (B.1) also for bounded domains. 
\end{rem}

\vspace{10pt}

\section*{Acknowledgements}

The author is grateful to Professors Thierry Gallay and Yoshikazu Giga for informing him of the papers \cite{SawadaTaniuchi}, \cite{Zelik}. He is also grateful to Professors Paolo Maremonti and Senjo Shimizu for sending him their Preprint. Finally the author expresses his gratitude to Professor Toshiaki Hishida for informing him of the paper \cite{BM1992} on fractional power estimates, related to Remarks 2.6 (iii). This work was partially supported by JSPS through the Grant-in-aid for Research Activity Start-up 15H06312, Young Scientists (B) 17K14217 and Scientific Research (B) 17H02853.

\vspace{10pt}

\section*{Conflict of Interest}
The author declares that he has no conflict of interest.

\vspace{10pt}

\bibliographystyle{plain}

\bibliography{ref}

\end{document}